\documentclass{tran-l}
\usepackage{amsmath}
\usepackage{amssymb}
\usepackage{hyperref}
\usepackage{graphicx}
\usepackage[all,cmtip]{xy}

 \newtheorem{theorem}{Theorem}[section]
 \renewcommand\thetheorem{\Alph{theorem}}
 \newtheorem{corollary}[theorem]{Corollary}
 \newtheorem{lemma}[theorem]{Lemma}
 
 \theoremstyle{definition}
 \newtheorem{definition}[theorem]{Definition}
 \theoremstyle{remark}
 \newtheorem{remark}{Remark}[section]
 
 \theoremstyle{example}
 \newtheorem{example}{Example}[section]
 \numberwithin{equation}{section}

\begin{document}

\title[Topological Equivalence and Morse Theory]
{An Application of Topological Equivalence to Morse Theory}

\author{LIZHEN QIN}

\address{Mathematics Department of Nanjing University, Nanjing, Jiangsu 210093, P.R.China}

\email{qinlz@nju.edu.cn}









\begin{abstract}
In a previous paper, under the assumption that the Riemannian
metric is special, the author proved some results about the moduli spaces
and CW structures arising from Morse theory. By virtue of
topological equivalence, this paper extends those results by
dropping the assumption on the metric.

In particular, we give a strong solution to the following classical
question: {\it Does a Morse function on a compact Riemannian
manifold give rise to a CW decomposition that is homeomorphic to
the manifold?}
\end{abstract}
\maketitle

\section{Introduction}\label{section_introduction}
In a previous paper \cite{qin}, the author proved some results on
moduli spaces and CW structures arising from Morse theory in the CF
case. By the CF case, we mean the Morse function satisfies the
Palais-Smale Condition (C) on a complete Hilbert-Riemannian manifold
and its critical points have finite indices (see \cite[def.\
2.6]{qin}). Those results include the manifold structure of the
compactified moduli spaces, orientation formulas, and the CW
structure on the underlying manifold. (See \cite{qin} for a detailed
description and a bibliography.)

Most results in \cite{qin} are based on the assumption of that the
Riemannian metric (or the negative gradient vector field) is
locally trivial (see Definition \ref{definition_locally_trivial}).
This means the vector field has the simplest form near each critical
point.

In this paper, by virtue of topological equivalence (see Definition
\ref{definition_topological_equivalence}), we shall extend those
results by dropping the above assumption provided that the Morse
function is proper. Here the underlying manifold has to be finite
dimensional but not necessarily compact.

In order to apply topological equivalence, based on the idea
outlined in the paper by Newhouse and Peixoto
\cite{newhouse_peixoto}, we shall prove the following main theorem
stated in Franks' paper \cite[prop.\ 1.6]{franks}.

\begin{theorem}\label{theorem_A}
Suppose $f$ is a Morse function on a compact manifold $M$. Suppose
$X$ is a negative gradient-like field for $f$ (see Definition
\ref{definition_gradient_like}), and $X$ satisfies transversality
(see Definition \ref{definition_transversality}). Then there is a
regular path between $X$ and $Y$ such that $Y$ is also a negative
gradient-like field for $f$. More importantly, $Y$ is locally
trivial. In particular, there is a topological equivalence between
$X$ and $Y$.
\end{theorem}

In Theorem \ref{theorem_A}, by a regular path, we mean a continuous
path of negative gradient-like vector fields in which each single
vector field on the path satisfies transversality. A precise version
of Theorem \ref{theorem_A} is Theorem \ref{theorem_regular_path}.

The importance of Theorem \ref{theorem_A} is that it can be combined
with the results of \cite{qin} to give an extension of those results
to more general metrics. In particular, we give a strong solution to
the following classical question which had been considered by Thom
(\cite{thom}), Bott (\cite[p.\ 104]{bott}) and Smale (\cite[p.\
197]{smale2}): {\it Does a Morse function on a compact Riemannian
manifold give rise to a CW decomposition that is homeomorphic to
the manifold such that its open cells are the unstable manifolds of
the negative gradient vector field?} A corollary of Theorem
\ref{theorem_cw_k(a)} gives the following answer which strengthens
the work in \cite{kalmbach2} and \cite{laudenbach} (see also Remark
\ref{remark_CW_compact}):

\begin{theorem}\label{theorem_B}
Suppose $f$, $M$ and $X$ are the same as those in Theorem
\ref{theorem_A}. Then the compactified unstable manifolds of $X$
give a CW decomposition that is homeomorphic to $M$. The open cells
of this CW complex are the unstable manifolds. Furthermore, the
characteristic maps have explicit formulas.
\end{theorem}

The following is the reason for making the extension of results in
\cite{qin}. There are at least two disadvantages of the locally
trivial metric assumed in \cite{qin}. Firstly, local triviality is
not a generic property. Sometimes, especially in the infinite
dimensional setting such as in Floer theory, it is not usually
the case that one can find a metric satisfying both the local
triviality and transversality conditions. Secondly, the assumption
of local triviality of the metric contradicts symmetry. Take for
example a homogeneous Riemannian manifold. If the metric is locally
trivial, then the curvature tensor must vanish near each critical
point. Since the metric is homogeneous, the curvature tensor must
vanish globally. Thus only a tiny class of homogeneous Riemannian
manifolds have this type of metric.

Actually, the local triviality assumption on the metric was made in
\cite{qin} exclusively because of the techniques employed there. The
theorems in \cite[thm.\ 3.3, 3.4, and 3.5]{qin} show that, under the
assumption of a locally trivial metric, the compactified moduli
spaces have smooth structures compatible with that of the underlying
manifold. However, the example in \cite[example 3.1]{qin} shows
that, if the metric is not locally trivial, there is no such
compatiblity (see also Remark \ref{remark_embed_m(p,q)}). Thus the
case of a locally trivial metric has several distinct features from
the general case. In fact, the proofs of \cite[thm.\ 3.7 and
3.8]{qin} rely heavily on the compatibility.

In this situation, it's natural to pose the following strategy for
obtaining results about Morse moduli spaces in the case of a general
metric. As a first step, we implement the subtle and technical
arguments in the special case. In the second and final step, we try
to convert the general case to the special case. The paper
\cite{qin} completes the first step. This paper achieves the second
one.

Franks' paper \cite[prop. 1.6]{franks} proposes an excellent idea to
reduce the general case to the special case as follows. The proof of
\cite[lem.\ 2]{newhouse_peixoto} claims that there exists a regular
path (i.e. each single vector field on the path satisfies
transversality) as the one stated in Theorem \ref{theorem_A}. Since
a negative gradient-like vector field satisfying transversality is
structurally stable, we get the topological equivalence in Theorem
\ref{theorem_A}, which converts the general vector field $X$ to the
locally trivial $Y$. (The argument in \cite{franks} also shows the
power of Theorem \ref{theorem_A}.)

However, the proof in \cite{newhouse_peixoto} does not provide sufficient details. It's well known that, for negative
gradient-like vector fields, transversality is preserved under small
$C^{1}$ perturbations. However, the vector fields certainly change
largely in the $C^{1}$ topology along the above path. \textit{How
can we guarantee the transversality?} Franks' paper \cite{franks}
refers the proof to \cite{newhouse_peixoto}, and the latter outlines
the construction of the path. Both \cite{franks} and
\cite{newhouse_peixoto} indicate that the $\lambda$-Lemma in
\cite{palis} verifies the transversality. Unfortunately, none of
them explain \textit{why} the $\lambda$-Lemma works in this setting.

The current paper supports the above idea in
\cite{newhouse_peixoto}. Precisely, following this idea, we shall
give a self-contained and detailed proof of Theorem
\ref{theorem_regular_path}. However, the statement of Theorem
\ref{theorem_regular_path} is slightly different from that in
\cite{newhouse_peixoto} such that it becomes better in the setting
of Morse theory. (Actually, the papers \cite{newhouse_peixoto} and
\cite{franks} emphasize the setting of dynamical systems. However,
our argument also proves the result in \cite{newhouse_peixoto}. See
Remark \ref{remark_path}.)

The main body of this paper consists of two parts. The first part,
Sections \ref{section_morse_lemma}-\ref{section_reduction_lemma},
consists of preparations for the application of topological
equivalence. The main theorems in it are Theorems
\ref{theorem_Morse} and \ref{theorem_regular_path}, which may be of
independent interest. The second part consists of the subsequent
sections and gives the application of topological equivalence.
Theorem \ref{theorem_topological_equivalence} shows that the
compactified moduli spaces are invariants of topological
equivalence, which is the base for our application. The theorems in
Sections \ref{section_property_moduli_spaces}-\ref{section_CW} are
extensions of those in \cite{qin}.


\renewcommand\thetheorem{\arabic{theorem}}
\numberwithin{theorem}{section}

\section{Preliminaries}\label{section_preliminaries}
In this section, we give some definitions, notation and elementary
results mostly used throughout the paper.

Suppose $M$ is a finite dimensional smooth manifold, and $f$ is a
proper Morse function on $M$. Let $M^{a,b}$ denote $f^{-1}([a,b])$.
Let $M^{a}$ denote $f^{-1}((-\infty, a])$.

\begin{definition}\label{definition_gradient_like}
A vector field $X$ is negative gradient-like for $f$ if $Xf(x) <
0$ when $x$ is not a critical point, and, near each critical point
$p$, $X$ is the negative gradient of $f$ for some metric.
\end{definition}

By Definition \ref{definition_gradient_like}, every negative gradient vector
field is obviously a negative gradient-like vector field. On the contrary,
Smale \cite[remark after thm.\ B]{smale1} gives the following fact
(see also \cite[lem.\ 7.12]{qin}).

\begin{lemma}\label{lemma_gradient}
Every negative gradient-like field of a Morse function $f$ is
actually a negative gradient field of $f$ for some metric.
\end{lemma}

By the Morse Lemma, there exists a local coordinate chart near a
critical point $p$ such that $p$ has the coordinate $(0,0)$, and the
function has the form
\begin{equation}\label{morse_chart}
  f(x_{1}, x_{2}) = f(p) - \frac{1}{2} \langle x_{1}, x_{1} \rangle
  + \frac{1}{2} \langle x_{2}, x_{2} \rangle
\end{equation}
in this chart. We call this chart a Morse chart.

\begin{definition}\label{definition_locally_trivial}
We say the metric of $M$ is trivial near $p$ if the metric of $M$
coincides with the standard metric of a Morse chart near $p$. In
other words, in this Morse chart, $- \nabla f$ has the simplest form
$- \nabla f(x_{1}, x_{2}) = (x_{1}, -x_{2})$. Similarly, we say a
negative gradient-like field $X$ is trivial near $p$ if $X(x_{1},
x_{2}) = (x_{1}, -x_{2})$ in a Morse chart. If the metric (or $X$)
is trivial near each critical point, we say this metric (or $X$) is
locally trivial.
\end{definition}

\begin{remark}
Some papers in the literature include the local triviality of $X$
into the definition of a gradient-like vector field. We follow the
style of \cite{smale1} and exclude it.
\end{remark}

We omit the proof of the following easy lemma which can be proved by inverse function theorem.
\begin{lemma}\label{lemma_flow_intersection}
Denote by $\phi_{t}(x)$ the flow generated by $X$ with initial value $x$ and time $t$. Suppose $a$ is a regular value of $f$. Suppose $\phi_{t}(x_{0}) \in f^{-1}(a)$ for some $t$. Then there exists a neighborhood $U$ of $x_{0}$ such that, for all $y \in U$, $\phi_{t(y)}(y) \in f^{-1}(a)$ for a unique $t(y)$. Furthermore, $t(y)$ and $\phi_{t(y)}(y)$ are smooth with respect to $y$.
\end{lemma}

\begin{definition}\label{definition_invariant_manifold}
Let $\phi_{t}(x)$ be the flow generated by  $X$ with initial value
$x$ and time $t$. Suppose $p$ is a critical point. Define the descending manifold
of $p$ as $\mathcal{D}(p) = \{ x \in M \mid \displaystyle \lim_{t
\rightarrow - \infty} \phi_{t}(x) = p \}$. Define the ascending
manifold of $p$ as $\mathcal{A}(p) = \{ x \in M \mid \displaystyle
\lim_{t \rightarrow + \infty} \phi_{t}(x) = p \}$. We call
$\mathcal{D}(p)$ and $\mathcal{A}(p)$ the invariant manifolds of
$p$. We also let $\mathcal{D}(p;X)$ denote $\mathcal{D}(p)$ and $\mathcal{A}(p;X)$
denote $\mathcal{A}(p)$ in order to indicate
the vector field $X$.
\end{definition}

Clearly, $\mathcal{D}(p)$ is the unstable manifold of $p$ with
respect to $X$, and $\mathcal{A}(p)$ is the stable manifold. They
are smoothly embedded open disks in $M$. Furthermore,
$\dim (\mathcal{D}(p)) = \mathrm{ind}(p)$, where $\mathrm{ind}(p)$
is the Morse index of $p$. (See e.g. \cite[lem.\ 3.8]{smale3})

\begin{definition}\label{definition_transversality}
We say that $X$ satisfies transversality if $\mathcal{D}(p)$ and
$\mathcal{A}(q)$ are transversal for all critical points $p$ and
$q$. For critical points $p$ and $q$, we say that $p$ and $q$ are
transversal if the invariant manifolds of $p$ are transverse to
those of $q$. Suppose $U$ is a subset of $M$, and these invariant
manifolds meet transversally at each point in $U$ (this includes the
case that they don't meet at that point). We say that $p$ and $q$
are transversal in $U$.
\end{definition}

The following lemma is obvious.

\begin{lemma}\label{lemma_local_transversal}
If $p$ and $q$ are transversal in $f^{-1}((a,b))$ and $p \in
f^{-1}((a,b))$, then $p$ and $q$ are transversal. If $p$ and $q$ are
transversal in $f^{-1}(a)$ and $f(q) < a < f(p)$, then $p$ and $q$
are transversal.
\end{lemma}

\begin{definition}\label{definition_partial_order}
Suppose $p$ and $q$ are critical points. Define $q \preceq p$ if
there exists a flow from $p$ to $q$. Define $q \prec p$ if
$q \preceq p$ and $q \neq p$.
\end{definition}

If $X$ satisfies transversality, then ``$\preceq$" is a partial
order on the set consisting of all critical points (see \cite[p.\
85, cor.\ 1]{palis_de}).

Now we introduce the definitions of topological conjugacy and
topological equivalence in dynamical systems. The reader is to be
forewarned that the definitions appearing in the literature are not
uniform. We follow the terminology of \cite[p.\ 26]{palis_de}. In
this paper, a topological conjugacy is a relation strictly stronger
than a topological equivalence. This is \textit{different} from the
definition in \cite{franks}. The ``topological conjugacy" in
\cite[p.\ 201]{franks} is actually the ``topological equivalence" in
this paper. Although a topological equivalence is good enough for
our application to Morse theory, we still introduce the notion of
topological conjugacy in order to make the statement of Theorem
\ref{theorem_regular_path} stronger.

\begin{definition}\label{definition_topological_equivalence}
Suppose $X_{i}$ ($i=1,2$) is a vector field on $M_{i}$ and
$\phi^{i}_{t}$ is the flow generated by $X_{i}$. Suppose $h: M_{1}
\rightarrow M_{2}$ is a homeomorphism. If $h \phi^{1}_{t} =
\phi^{2}_{t} h$, then we call $h$ a topological conjugacy between
$X_{1}$ and $X_{2}$. If $h$ maps the orbits of $X_{1}$ to the orbits
of $X_{2}$ and $h$ preserves the directions of orbits, then we call
$h$ a topological equivalence between $X_{1}$ and $X_{2}$.
\end{definition}

\begin{remark}
In dynamical systems, people usually consider the topological
equivalence (or conjugacy) of vector fields on one manifold $M$,
i.e. $M_{1} = M_{2}$ in Definition
\ref{definition_topological_equivalence}. However, it seems
beneficial for topology to allow that $M_{1}$ is not diffeomorphic
to $M_{2}$. For example, choose a standard sphere $S^{n}$ and a
twist sphere $\Sigma^{n}$. Let $f_{1}$ and $f_{2}$ be the height
functions on $S^{n}$ and $\Sigma^{n}$ respectively. We can define a
topological conjugacy between $- \nabla f_{1}$ and $- \nabla f_{2}$
as follows. Choose a homeomorphism (or even a diffeomorphism) $h_{0}
: S^{n-1} \rightarrow \Sigma^{n-1}$, where $S^{n-1}$ and
$\Sigma^{n-1}$ are the equators of $S^{n}$ and $\Sigma^{n}$
respectively. Define $h$ such that $h \phi^{1}_{t}(x) = \phi^{2}_{t}
h_{0} (x)$ for all $x \in S^{n-1}$, and $h$ maps the maximum
(minimum) point to the maximum (minimum) point. Clearly, this
topological conjugacy $h$ recovers the Alexander trick.
\end{remark}

\section{A Strengthened Morse Lemma}\label{section_morse_lemma}
In this section, we shall present a Strengthened Morse Lemma which
is useful for the proof of Theorem \ref{theorem_regular_path} (See
Remarks \ref{remark_morse_path} and \ref{remark_path}).

Suppose $H$ is a Hilbert space with inner product $\langle \cdot,
\cdot \rangle$, and $U$ is an open subset of $H$. Define a smooth
Riemannian metric (or smooth metric for brevity) on $U$ in the usual
sense. In other words, for each $x \in U$, assign a bounded symmetric
positive definite linear operator $A(x)$ such that $A(x)$ is a
smooth function of $x$. For any $v$ and $w$ in $T_{x} U = H$, define
$\langle v, w \rangle_{G(x)} = \langle A(x) v, w \rangle$.

\begin{theorem}[Strengthened Morse Lemma]\label{theorem_Morse}
Suppose $H$ is a Hilbert space, $U$ is an open neighborhood of $0
\in H$. Suppose $f$ is a smooth Morse function on $U$ with a
critical point $0$, and $G$ is a smooth metric on $U$. Let $-
\nabla_{G} f$ be the negative gradient of $f$ with respect to $G$,
and $\phi_{t}$ be the flow generated by $- \nabla_{G} f$. Suppose $H
= H_{1} \oplus H_{2}$, where $H_{1}$ and $H_{2}$ are the negative
and positive spectral spaces of $\nabla_{G}^{2} f(0)$ respectively.
Then there exist an $\epsilon >0$, an open neighborhood $V$ of $0$ such that $V
\subseteq U$, $B_{1} = \{ x_{1} \in H_{1} \mid \|x_{1}\| < \epsilon
\}$, $B_{2} = \{ x_{2} \in H_{2} \mid \|x_{2}\| < \epsilon \}$, and
a diffeomorphism $h: B_{1} \times B_{2} \rightarrow V$ such that the
following holds. We have
\begin{equation}\label{theorem_Morse_1}
  h^{*}f(x_{1},x_{2}) = f(0) - \frac{1}{2} \langle x_{1}, x_{1} \rangle
  + \frac{1}{2} \langle x_{2}, x_{2} \rangle,
\end{equation}
\begin{eqnarray*}
   h(B_{1})  & = &  \mathcal{D}_{V}(0; - \nabla_{G} f) = \{ x \in V \mid \phi((-\infty, 0], x) \subseteq V \} \\
   & = & \left \{ x \in V \mid \phi((-\infty, 0], x) \subseteq V, \lim_{t \rightarrow - \infty}
   \phi(t, x) = 0 \right \},
\end{eqnarray*}
and
\begin{eqnarray*}
   h(B_{2})  & = &  \mathcal{A}_{V}(0; - \nabla_{G} f) = \{ x \in V \mid \phi([0, +\infty), x) \subseteq V \} \\
   & = & \left \{ x \in V \mid \phi([0, +\infty), x) \subseteq V, \lim_{t \rightarrow + \infty}
   \phi(t, x) = 0 \right \}.
\end{eqnarray*}
\end{theorem}

Before proving it, we explain the statement of Theorem
\ref{theorem_Morse}. In this theorem, $\mathcal{D}_{V}(0; -
\nabla_{G} f)$ is the local unstable (descending) manifold of $0$ in
the neighborhood $V$, and $\mathcal{A}_{V}(0; - \nabla_{G} f)$ is
the local stable (ascending) manifold. They certainly depend on the
metric. The classical Morse Lemma shows that, by a coordinate
transformation $h$, we get a new chart (we call it a Morse Chart)
such that the function has the form (\ref{theorem_Morse_1}) in it.
Theorem \ref{theorem_Morse} tells us more: No matter what the metric
is, there exists a Morse chart such that the local invariant
manifolds are standard in it. (Figure \ref{figure_morse_lemma}
illustrates this strengthened Morse chart, where the arrows indicate
the directions of the flows.) This makes three objects, i.e. the
function, the local invariant manifolds, and the coordinate chart
fit well. In short, Theorem \ref{theorem_Morse} strengthens the
classical Morse Lemma by taking the dynamical system into account.

\begin{figure}[!htbp]
\centering
\includegraphics[scale=0.4]{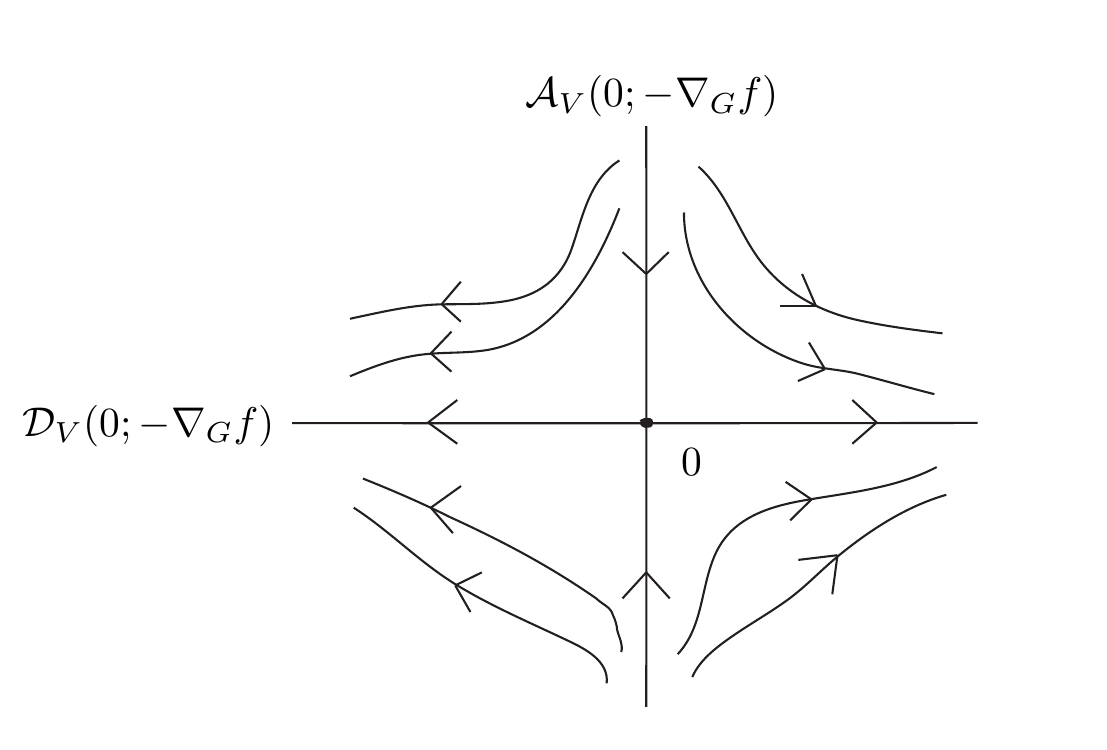} \caption{Strengthened Morse Chart}
\label{figure_morse_lemma}
\end{figure}

\begin{proof}[Proof of Theorem \ref{theorem_Morse}]
We know that $\phi_{1}$ is a smooth map defined on $U_{0}$ with a
hyperbolic fixed point $0$, where $U_{0}$ is a neighborhood of $0$.
By the Local Invariant Manifold Theorem (see \cite{irwin1} and
\cite[thm.\ 28]{irwin2}), shrinking $U_{0}$ suitably, there exists a
diffeomorphism $h_{1}: \widetilde{B}_{1} \times \widetilde{B}_{2}
\rightarrow U_{0}$ such that
\begin{eqnarray*}
   h_{1}(\widetilde{B}_{1})  & = &  \mathcal{D}_{U_{0}}(0; \phi_{1}) = \{ x \in U_{0} \mid \forall n \leq 0, (\phi_{1})^{n}(x) \in U_{0} \} \\
   & = & \left \{ x \in U_{0} \mid \forall n \leq 0, (\phi_{1})^{n}(x) \in U_{0}, \lim_{n \rightarrow - \infty}
   (\phi_{1})^{n}(x) = 0 \right \},
\end{eqnarray*}
$h_{1}(\widetilde{B}_{2}) = \mathcal{A}_{U_{0}}(0; \phi_{1})$, and $h_{1} (0) =0$.
Here the definition of $\mathcal{A}_{U_{0}}(0; \phi_{1})$ is similar
to that of $\mathcal{D}_{U_{0}}(0; \phi_{1})$, $\widetilde{B}_{1} \times \widetilde{B}_{2} \subseteq H_{1} \times
H_{2}$, and $\widetilde{B}_{1} \times \widetilde{B}_{2}$ is an open neighborhood of $(0,0)$.

Clearly, $h_{1}^{*}f|_{\widetilde{B}_{1}}$ and
$h_{1}^{*}f|_{\widetilde{B}_{2}}$ are Morse functions on
$\widetilde{B}_{1}$ and $\widetilde{B}_{2}$ respectively. By the
Morse Lemma, composing $h_{1}$ with a diffeomorphism if necessary,
we may assume that
\[
  h_{1}^{*}f|_{\widetilde{B}_{1}} = f(0) - \frac{1}{2} \langle x_{1}, x_{1}
  \rangle, \qquad
  h_{1}^{*}f|_{\widetilde{B}_{2}} = f(0) + \frac{1}{2} \langle x_{2}, x_{2}
  \rangle.
\]

Define
\[
  R(x) = h_{1}^{*}f(x) - \left( f(0) - \frac{1}{2} \langle x_{1}, x_{1} \rangle
  + \frac{1}{2} \langle x_{2}, x_{2} \rangle \right).
\]
Here $x=(x_{1},x_{2})$. Denote the differential of $R$ with order
$n$ by $D^{n} R$. Then $R(x_{1},0) \equiv 0$ and $R(0,x_{2}) \equiv
0$. In addition, since $D f(0) =0$, for any $v_{1} \in H_{1}$ and $v_{2} \in H_{2}$, we
have
\begin{eqnarray*}
   D^{2}(h_{1}^{*}f)(0)(v_{1}, v_{2})  & = &  D^{2}f(Dh_{1} (0) \cdot v_{1}, Dh_{1} (0) \cdot v_{2}) \\
   & = & \langle \nabla_{G}^{2} f(0) Dh_{1} (0) \cdot v_{1}, Dh_{1} (0) \cdot v_{2}
   \rangle_{G(0)}.
\end{eqnarray*}
By the Local Invariant Manifold Theorem again, we have the tangent spaces
\[
T_{0} \mathcal{D}_{U_{0}}(0; \phi_{1}) =H_{1} \qquad \text{and} \qquad T_{0} \mathcal{A}_{U_{0}}(0; \phi_{1}) =H_{2}.
\]
Therefore $Dh_{1} (0) \cdot v_{1} \in H_{1}$ and $Dh_{1} (0) \cdot v_{2} \in
H_{2}$. By the assumption, $\nabla_{G}^{2} f(0)$ is symmetric with respect to $G(0)$,
and $H_{1}$ and $H_{2}$ are negative and positive spectral spaces of
$\nabla_{G}^{2} f(0)$ respectively. Thus
$D^{2}(h_{1}^{*}f)(0)(v_{1}, v_{2}) = 0$. We infer $D^{2}R(0)(v_{1},
v_{2}) = 0$ and $D^{2}_{1,2}R(0) = 0$, where $D^{2}_{1,2} R (0)$ is the restriction of $D^{2}R(0)$ on $H_{1} \times H_{2}$.

Now we have
\begin{eqnarray*}
   &   & R(x_{1}, x_{2}) \\
   & = & R(x_{1},0) + \int_{0}^{1} \frac{d}{dt} R(x_{1}, t x_{2})dt \\
   & = & \int_{0}^{1} D_{2} R(x_{1}, t x_{2})dt \cdot x_{2} \quad
   (\text{because $R(x_{1},0) = 0)$} \\
   & = & \int_{0}^{1} \left[ D_{2} R(0, t x_{2}) + \int_{0}^{1} \frac{d}{ds} D_{2} R(s x_{1}, t x_{2}) ds
   \right] dt \cdot x_{2} \\
   & = & \int_{0}^{1} \int_{0}^{1} D^{2}_{1,2} R(s x_{1}, t x_{2})
   ds dt (x_{1}, x_{2}) \quad (\text{because $D_{2}R(0, t x_{2}) = 0$}) \\
   & = & \int_{0}^{1} \int_{0}^{1} \left[ D^{2}_{1,2} R(0, 0) + \int_{0}^{1} \frac{d}{d \tau}
   D^{2}_{1,2} R(\tau s x_{1}, \tau t x_{2}) d \tau \right] ds dt (x_{1}, x_{2}) \\
   & = & \int_{0}^{1} \int_{0}^{1} \int_{0}^{1} s D^{3}_{1,1,2} R(\tau s x_{1}, \tau t
   x_{2}) d \tau ds dt (x_{1}, x_{1}, x_{2}) \quad (\text{because $D^{2}_{1,2}R(0) = 0$}) \\
   &   & + \int_{0}^{1} \int_{0}^{1} \int_{0}^{1} t D^{3}_{1,2,2} R(\tau s x_{1}, \tau t
   x_{2}) d \tau ds dt (x_{1}, x_{2}, x_{2}).
\end{eqnarray*}
Since $D^{3} R$ is a bounded symmetric multilinear form, there exist bounded
symmetric operators $R_{1}(x)$ and $R_{2}(x)$ on $H_{1}$ and $H_{2}$
respectively such that, for any $v_{1}$ and $w_{1}$ in $H_{1}$,
\[
  \int_{0}^{1} \int_{0}^{1} \int_{0}^{1} s D^{3}_{1,1,2} R(\tau s x_{1}, \tau t
   x_{2}) d \tau ds dt (v_{1}, w_{1}, x_{2}) = \frac{1}{2} \langle R_{1}(x_{1},
   x_{2}) v_{1}, w_{1} \rangle;
\]
and, for any $v_{2}$ and $w_{2}$ in $H_{2}$,
\[
  \int_{0}^{1} \int_{0}^{1} \int_{0}^{1} t D^{3}_{1,2,2} R(\tau s x_{1}, \tau t
   x_{2}) d \tau ds dt (x_{1}, v_{2}, w_{2}) = \frac{1}{2} \langle R_{2}(x_{1},
   x_{2}) v_{2}, w_{2} \rangle.
\]
Here $R_{1}(x)$ and $R_{2}(x)$ are smooth with respect to $x$.

Clearly, $R_{1}(0) = 0$, $R_{2}(0) = 0$, and
\[
  h_{1}^{*}f(x) = f(0) - \frac{1}{2} \langle (I - R_{1})(x) x_{1}, x_{1} \rangle
  + \frac{1}{2} \langle (I + R_{2})(x) x_{2}, x_{2} \rangle.
\]
Since $I-R_{1}$ and $I+R_{2}$ are symmetric, $I-R_{1}(0)=I$ and
$I+R_{2}(0)=I$, shrinking $\widetilde{B}_{1} \times
\widetilde{B}_{2}$ if necessary, we have $I-R_{1}(x) = C_{1}(x)^{2}$
and $I+R_{2}(x) = C_{2}(x)^{2}$. Here $C_{1}(x)$ and $C_{2}(x)$ are bounded
symmetric and positive definite operators on $H_{1}$ and $H_{2}$
respectively, and they are smooth functions of $x$. Thus
\[
  h_{1}^{*}f(x) = f(0) - \frac{1}{2} \langle C_{1}(x) x_{1}, C_{1}(x) x_{1} \rangle
  + \frac{1}{2} \langle C_{2}(x) x_{2}, C_{2}(x) x_{2} \rangle.
\]
Define $h_{2}: \widetilde{B}_{1} \times \widetilde{B}_{2}
\rightarrow H_{1} \times H_{2}$ by $h_{2}(x) = (C_{1}(x) x_{1},
C_{2}(x) x_{2})$. Then $h_{2}(\widetilde{B}_{1}) \subseteq H_{1}$
and $h_{2}(\widetilde{B}_{2}) \subseteq H_{2}$. Clearly, $C_{1}(0)=I$ and $C_{2}(0)=I$. Since
\[
Dh_{2} (x; v_{1}, v_{2}) = (C_{1}(x) v_{1} + DC_{1} (x; v_{1}, v_{2}) x_{1},  C_{2}(x) v_{2} + DC_{2} (x; v_{1}, v_{2}) x_{2}),
\]
we have
\begin{eqnarray*}
Dh_{2} (0; v_{1}, v_{2}) & = & (C_{1}(0) v_{1} + DC_{1} (0; v_{1}, v_{2}) \cdot 0,  C_{2}(0) v_{2} + DC_{2} (0; v_{1}, v_{2}) \cdot 0) \\
& = & (v_{1}, v_{2})
\end{eqnarray*}
and then $Dh_{2}(0) =I$. There exists $B_{1} \times B_{2} \subseteq H_{1} \times H_{2}$ such that $h_{2}^{-1}$ exists and is smooth on
$B_{1} \times B_{2}$. Then we get
\[
  (h_{1} \circ h_{2}^{-1})^{*} f(x) = f(0) - \frac{1}{2} \langle x_{1}, x_{1} \rangle
  + \frac{1}{2} \langle x_{2}, x_{2} \rangle.
\]
Define $h = h_{1} \circ h_{2}^{-1}$ and $V = h(B_{1}
\times B_{2})$.

We see that $h^{-1}(\mathcal{D}_{U_{0}}(0; \phi_{1})) = B_{1}$ and
$h^{-1}(\mathcal{A}_{U_{0}}(0; \phi_{1})) = B_{2}$. It is straightforward to prove that
$h(B_{1}) = \mathcal{D}_{V}(0; - \nabla_{G} f)$ and $h(B_{2}) =
\mathcal{A}_{V}(0; - \nabla_{G} f)$.
\end{proof}

\section{A Regular Path}\label{section_regular_path}
As mentioned in the Introduction, the purpose of this section is to
present a detailed proof of Theorem \ref{theorem_regular_path} in
order to support an idea outlined in \cite[lem.\
2]{newhouse_peixoto}. In this proof, Lemma \ref{lemma_inclination}
plays a key role.

\begin{theorem}[Regular Path]\label{theorem_regular_path}
Suppose $f$ is a Morse function on a compact manifold $M$. Suppose
$X$ is a negative gradient-like field for $f$, and $X$ satisfies
transversality. Then there is a continuous path $\mathcal{Y}: [0,1]
\rightarrow \mathfrak{X}^{\infty}(M)$ such that, for all $s \in
[0,1]$, $\mathcal{Y}_{s}$ is a negative gradient-like field for $f$,
$\mathcal{Y}_{s}$ satisfies transversality, $\mathcal{Y}_{0} = X$
and $\mathcal{Y}_{1}$ is locally trivial. In particular, there
exists a topological conjugacy $h$ between $X$ and $\mathcal{Y}_{1}$
such that $h(p)=p$ for each critical point $p$. Here
$\mathfrak{X}^{\infty}(M)$ is the set with the Whitney $C^{\infty}$
topology consisting of $C^{\infty}$ vector fields on $M$.
\end{theorem}

We call a continuous path of negative gradient-like vector fields
$\mathcal{Y}: [a,b] \rightarrow \mathfrak{X}^{\infty}(M)$ a regular
path if $\mathcal{Y}_{s}$ satisfies transversality for all $s$.

We need the following classical Comparison Theorem for ODEs (see
\cite[p.\ 96]{walter}).

\begin{theorem}[well-known]\label{theorem_comparison}
Suppose $F(t,x)$ is a Lipschitz continuous function defined on
$[t_{0}, t_{1}] \times [a,b]$. Let $x(t)$ be the solution of the
equation $\dot{x} = F(t,x)$ with $x(t_{0}) = x_{0}$. Suppose $y(t)$
is a $C^{1}$ function defined on $[t_{0}, t_{1}]$ with $y(t_{0}) =
x_{0}$. Then
\begin{enumerate}
\item if $\dot{y} \leq F(t,y)$, then $y(t) \leq x(t)$ on $[t_{0},
t_{1}]$;

\item  if $\dot{y} \geq F(t,y)$, then $y(t) \geq x(t)$ on $[t_{0},
t_{1}]$.
\end{enumerate}
\end{theorem}

Suppose $H = H_{1} \oplus H_{2}$ is a Hilbert space, $v = (v_{1},
v_{2}) \in H$, $v_{1} \neq 0$, and $\lambda =
\frac{\|v_{2}\|}{\|v_{1}\|}$. We call $\lambda$ the inclination of
$v$ with respect to $H_{1}$. Suppose $L$ is a closed subspace of
$H$, and $P: H \rightarrow H_{1}$ is the projection. If $P: L
\rightarrow P(L)$ is a topological linear isomorphism, then there
exists a bounded linear operator $A: P(L) \rightarrow H_{2}$ such
that $L$ is the graph of $A$, i.e., for any $v \in L$, we have $v =
(v_{1}, A v_{1})$, where $v_{1} \in P(L)$. We call the supremum of
the inclinations of all non-zero vectors in $L$ the inclination of
$L$ with respect to $H_{1}$. Clearly, the inclination of $L$ equals,
$\|A\|$, the norm of $A$.

Suppose $H$, $H_{1}$ and $H_{2}$ are Hilbert spaces as above.
Suppose $A_{0}$ and $A_{1}$ are bounded linear operators on $H_{1}$, and $B$
is a bounded linear operator on $H_{2}$. Suppose further there exist positive numbers
$\alpha_{0} > 0$, $\alpha_{1} > 0$ and $\beta > 0$ such that
\begin{equation}\label{first_operator}
  \alpha_{0} \langle w, w \rangle \leq \langle A_{i} w, w \rangle
  \leq \alpha_{1} \langle w, w \rangle \qquad (i=0,1),
\end{equation}
and
\begin{equation}\label{second_operator}
  \beta \langle w, w \rangle \leq \langle B w, w \rangle.
\end{equation}

Let $\rho$ be a smooth bump function on $(-\infty, +\infty)$ such
that $0 \leq \rho \leq 1$, $\rho(s) \equiv 1$ when $s \leq
\frac{1}{2}$, and $\rho(s) \equiv 0$ when $s \geq 1$. Define
$\rho_{r}(s) = \rho (\frac{s}{r})$ for $r > 0$. For convenience, we
denote $\rho_{r}(\|x_{i}\|)$ by $\rho_{r}(x_{i})$, where $x_{i} \in
H_{i}$.

Define a smooth vector field $X_{r}$ on $H$ by
\begin{equation}\label{equation_x_r}
  X_{r} (x_{1}, x_{2}) = (\rho_{r}(x_{1}) \rho_{r}(x_{2}) A_{0} x_{1} +
  [1- \rho_{r}(x_{1}) \rho_{r}(x_{2})] A_{1} x_{1}, - B x_{2}).
\end{equation}
Denote the flow generated by $X_{r}$ by $\phi_{t}(x_{1}, x_{2})$.
For a fixed $t$, $\phi_{t}$ is a diffeomorphism, thus $D \phi_{t}$
acts on the tangent vectors at each point $(x_{1}, x_{2})$, where $D
\phi_{t}$ is the differential of $\phi_{t}$ with respect to $x =
(x_{1},x_{2})$.

\begin{lemma}\label{lemma_inclination}
For any $\epsilon > 0$, there exists $\delta >0$ such that the
following holds. For any $r>0$ and $v \in H$, if the inclination of
$v$ with respect to $H_{1}$ is less than $\delta$, then we have the
inclination of $D \phi_{t} \cdot v$ with respect to $H_{1}$ is less
than $\epsilon$ for all $t \geq 0$. Here $\delta$ only depends on
$\alpha_{0}$, $\alpha_{1}$, $\beta$ and $\epsilon$, and $\delta$ is
independent of $r$.
\end{lemma}
\begin{proof}
The flow $\phi_{t} = (\phi^{1}_{t}, \phi^{2}_{t})$ satisfies the
following ordinary differential equation
\[
   \left\{ \begin{array} {rcl}
             \dot{{\phi}_{t}}^{1}  &  =  &  \rho_{r}(\phi_{t}^{1}) \rho_{r}(\phi_{t}^{2}) A_{0} \phi_{t}^{1} +
                                    [1- \rho_{r}(\phi_{t}^{1}) \rho_{r}(\phi_{t}^{2})] A_{1} \phi_{t}^{1}, \\
             \dot{{\phi}_{t}}^{2}  &  =  &  - B \phi_{t}^{2}.
           \end{array}
   \right.
\]
Denote $\rho_{r}(\phi_{t}^{1}) \rho_{r}(\phi_{t}^{2}) A_{0} + [1- \rho_{r}(\phi_{t}^{1}) \rho_{r}(\phi_{t}^{2})] A_{1}$ by $A(\phi_{t}^{1}, \phi_{t}^{2})$. We have
\[
  \frac{d}{dt} \langle \phi_{t}^{1}, \phi_{t}^{1} \rangle = 2 \langle
  \dot{{\phi}_{t}}^{1}, \phi_{t}^{1} \rangle = 2 \langle A(\phi_{t}^{1},
\phi_{t}^{2}) \phi_{t}^{1}, \phi_{t}^{1} \rangle.
\]
By (\ref{first_operator}), we have
\[
  0 \leq 2 \alpha_{0} \langle \phi_{t}^{1}, \phi_{t}^{1} \rangle
  \leq \frac{d}{dt} \langle \phi_{t}^{1}, \phi_{t}^{1} \rangle
  \leq 2 \alpha_{1} \langle \phi_{t}^{1}, \phi_{t}^{1} \rangle.
\]
Thus $\| \phi_{t}^{1} \|$ is increasing, and by Theorem
\ref{theorem_comparison}, we have
\begin{equation}\label{lemma_inclination_1}
  e^{\alpha_{0} t } \| \phi^{1}_{0} \|  \leq  \| \phi^{1}_{t} \|
  \leq  e^{\alpha_{1} t } \| \phi^{1}_{0} \|.
\end{equation}
Similarly, $\| \phi_{t}^{2} \|$ is decreasing, $\phi^{2}_{t} = e^{-Bt} \phi^{2}_{0}$, and $\| \phi^{2}_{t} \|  \leq  e^{- \beta t } \| \phi^{2}_{0} \|$.

Let $\mathbb{D}_{1}(r) = \{ x_{1} \in H_{1} \mid \| x_{1} \| < r
\}$, and $\mathbb{D}_{2}(r) = \{ x_{2} \in H_{2} \mid \| x_{2} \| <
r \}$. Clearly, $A(x_{1}, x_{2})|_{H- (\mathbb{D}_{1}(r) \times
\mathbb{D}_{2}(r))} = A_{1}$, and $A(x_{1},
x_{2})|_{\overline{\mathbb{D}_{1}(\frac{r}{2}) \times
\mathbb{D}_{2}(\frac{r}{2})}} = A_{0}$. Denote
$\overline{\mathbb{D}_{1}(r) \times \mathbb{D}_{2}(r)} -
(\mathbb{D}_{1}(\frac{r}{2}) \times \mathbb{D}_{2}(\frac{r}{2}))$ by
$E(r)$. (This is illustrated by Figure \ref{figure_domain_E_r}. The shadowed part is $E(r)$. The arrows indicate the directions of flows.) When $\phi([0,t], x)$ is out of $E(r)$, we have $\phi(t,x) =
(e^{A_{i}t} x_{1}, e^{-Bt} x_{2})$, and
\[
D \phi_{t} = \left(
\begin{array}{cc}
e^{A_{i}t} & 0 \\
0 & e^{-Bt}
\end{array}
\right).
\]
Since $\| e^{A_{i}t} w \| \geq \|w\|$ and $\| e^{-Bt} w \| \leq
\|w\|$ for $t \geq 0$, we have that the inclination of $D \phi_{t}
\cdot v$ is decreasing when $t$ is increasing. Thus it suffices to
control the variation of the inclination when $\phi_{t}(x)$ passes
through $E(r)$.

\begin{figure}[!htbp]
\centering
\includegraphics[scale=0.4]{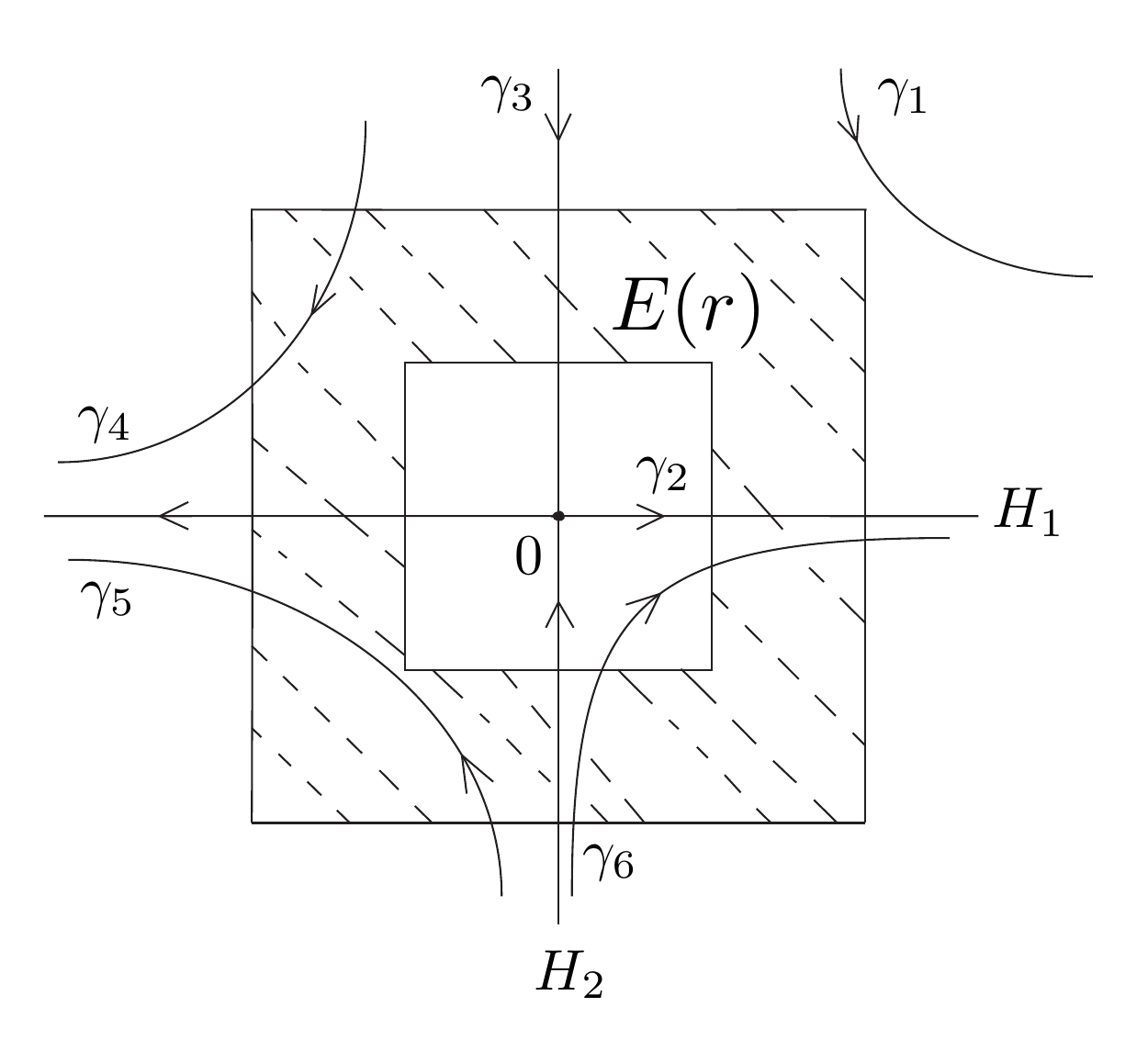} \caption{Domanin $E(r)$}
\label{figure_domain_E_r}
\end{figure}

Suppose $t \geq 0$ and $\| \phi^{1}_{t} \| = 2 \| \phi^{1}_{0} \|$,
then by (\ref{lemma_inclination_1}), we have $t \leq \frac{\ln
2}{\alpha_{0}}$. Similarly, if $\| \phi^{2}_{t} \| = \frac{1}{2} \|
\phi^{2}_{0} \|$, then $t \leq \frac{\ln 2}{\beta}$. Since $\|
\phi^{1}_{t} \|$ is increasing and $\| \phi^{2}_{t} \|$ is
decreasing, we infer that $\phi_{t}$ enters $E(r)$ at most twice,
and the time for it to stay in $E(r)$ is no more than
\begin{equation}\label{lemma_inclination_2}
  T = \frac{\ln 2}{\alpha_{0}} + \frac{\ln 2}{\beta}.
\end{equation}
(Figure \ref{figure_domain_E_r} illustrates this. The trajectory $\gamma_{1}$ and the constant trajectory at $0$ do not enter $E(r)$. The trajectory $\gamma_{2}$ is in $H_{1}$. Both $\gamma_{2}$ and $\gamma_{4}$ enter $E(r)$ once, and the time for them to stay in $E(r)$ is no more than $\frac{\ln 2}{\alpha_{0}}$. The trajectory $\gamma_{3}$ is in $H_{2}$, it enters $E(r)$ once and stays for no more than $\frac{\ln 2}{\beta}$. The trajectory $\gamma_{5}$ enters $E(r)$ once and stays for no more than $\frac{\ln 2}{\alpha_{0}} + \frac{\ln 2}{\beta}$. The trajectory $\gamma_{6}$ enters $E(r)$ twice, the first time for it to stay is no more than $\frac{\ln 2}{\beta}$, and the second time is no more than $\frac{\ln 2}{\alpha_{0}}$.)

Suppose $\phi([0,t],x) \subset E(r)$, we have $0 \leq t \leq T$.
Since $\phi^{2}_{t}(x) = e^{-Bt}x_{2}$, we have
\begin{equation}\label{lemma_inclination_3}
  D_{1} \phi^{2}_{t} = 0, \qquad D_{2} \phi^{2}_{t} = e^{-Bt},
  \qquad \text{and} \quad \| D_{2} \phi^{2}_{t} \cdot w \| \leq \| w \|.
\end{equation}

Since
\[
  \dot{{\phi}_{t}}^{1} = A(\phi_{t}^{1}, e^{-Bt}x_{2}) \phi_{t}^{1},
\]
we have
\[
  D_{1} \dot{{\phi}_{t}}^{1} \cdot w = A(\phi_{t}^{1}, \phi_{t}^{2}) (D_{1}
 \phi_{t}^{1} \cdot w) + D \rho_{r}(\phi_{t}^{1}) (D_{1} \phi_{t}^{1} \cdot w)
 \rho_{r}(\phi_{t}^{2}) (A_{0}-A_{1}) \phi_{t}^{1}.
\]
Thus
\begin{eqnarray*}
   \frac{d}{dt} \langle D_{1} \phi_{t}^{1} \cdot w, D_{1} \phi_{t}^{1} \cdot
   w \rangle & = & 2 \langle D_{1} \dot{{\phi}_{t}}^{1} \cdot w,
   D_{1} \phi_{t}^{1} \cdot w \rangle \\
   & = & 2 \langle A(\phi_{t}^{1}, \phi_{t}^{2}) (D_{1} \phi_{t}^{1} \cdot w), D_{1} \phi_{t}^{1} \cdot
   w \rangle \\
   &   & + 2 \langle D \rho_{r}(\phi_{t}^{1}) (D_{1} \phi_{t}^{1} \cdot w)
   \rho_{r}(\phi_{t}^{2}) (A_{0}-A_{1}) \phi_{t}^{1}, D_{1} \phi_{t}^{1} \cdot w \rangle .
\end{eqnarray*}
Clearly, $D \rho_{r}(\phi_{t}^{1}) = O(r^{-1})$, and $\| \phi_{t}^{1} \|
\leq r$ when $D \rho_{r}(\phi_{t}^{1}) \neq 0$. So there exists a
constant $C_{1} > 0$ which is independent of $r$ such that
\[
  | \langle D \rho_{r}(\phi_{t}^{1}) (D_{1} \phi_{t}^{1} \cdot w)
  \rho_{r}(\phi_{t}^{2}) (A_{0}-A_{1}) \phi_{t}^{1}, D_{1} \phi_{t}^{1} \cdot w \rangle
  | \leq C_{1} \| D_{1} \phi_{t}^{1} \cdot w \|^{2}.
\]
Combining the above inequality with (\ref{first_operator}), we get
\[
  -2 C_{1} \langle D_{1} \phi_{t}^{1} \cdot w, D_{1} \phi_{t}^{1} \cdot w
  \rangle \leq \frac{d}{dt} \langle D_{1} \phi_{t}^{1} \cdot w, D_{1} \phi_{t}^{1} \cdot
  w \rangle.
\]
Since $D_{1} \phi^{1}_{0} = I$ and $\| D_{1} \phi^{1}_{0} \cdot w \|
= \| w \|$, by Theorem \ref{theorem_comparison}, we have
\begin{equation}\label{lemma_inclination_4}
   \| D_{1} \phi^{1}_{t} \cdot w \| \geq e^{-C_{1} t} \|w\| \geq e^{-C_{1} T} \|w\|.
\end{equation}

Similarly, we have
\begin{eqnarray*}
   \frac{d}{dt} \langle D_{2} \phi_{t}^{1} \cdot w, D_{2} \phi_{t}^{1} \cdot w \rangle
   & = & 2 \langle A(\phi_{t}^{1}, \phi_{t}^{2}) (D_{2} \phi_{t}^{1} \cdot w), D_{2} \phi_{t}^{1} \cdot
   w \rangle \\
   &   & + 2 \langle D \rho_{r}(\phi_{t}^{1}) (D_{2} \phi_{t}^{1} \cdot w)
   \rho_{r}(\phi_{t}^{2}) (A_{0}-A_{1}) \phi_{t}^{1}, D_{2} \phi_{t}^{1} \cdot w \rangle \\
   &   & + 2 \langle \rho_{r}(\phi_{t}^{1}) D \rho_{r}(\phi_{t}^{2}) e^{-Bt} w (A_{0}-A_{1}) \phi_{t}^{1},
   D_{2} \phi_{t}^{1} \cdot w \rangle,
\end{eqnarray*}
and
\[
  | \langle D \rho_{r}(\phi_{t}^{1}) (D_{2} \phi_{t}^{1} \cdot w)
  \rho_{r}(\phi_{t}^{2}) (A_{0}-A_{1}) \phi_{t}^{1}, D_{2} \phi_{t}^{1} \cdot w \rangle
  | \leq C_{1} \| D_{2} \phi_{t}^{1} \cdot w \|^{2}.
\]
In addition, $\rho_{r}(\phi_{t}^{1}) D \rho_{r}(\phi_{t}^{2}) = O(r^{-1})$,
and $\| \phi_{t}^{1} \| \leq r$ when $\rho_{r}(\phi_{t}^{1}) D
\rho_{r}(\phi_{t}^{2}) \neq 0$. So there exists $C_{2} > 0$ which is
independent of $r$ such that
\begin{eqnarray*}
   &      & 2 | \langle \rho_{r}(\phi_{t}^{1}) D \rho_{r}(\phi_{t}^{2}) e^{-Bt} w (A_{0}-A_{1}) \phi_{t}^{1},
   D_{2} \phi_{t}^{1} \cdot w \rangle | \\
   & \leq &  2 C_{2} \| D_{2} \phi_{t}^{1} \cdot w \| \|w\| \leq C_{2} \| D_{2} \phi_{t}^{1} \cdot w
   \|^{2} + C_{2} \|w\|^{2}.
\end{eqnarray*}
Thus by (\ref{first_operator}), we infer
\[
   \frac{d}{dt} \langle D_{2} \phi_{t}^{1} \cdot w, D_{2} \phi_{t}^{1} \cdot w
   \rangle \leq   (2 \alpha_{1} + 2 C_{1} + C_{2}) \langle D_{2} \phi_{t}^{1} \cdot w, D_{2} \phi_{t}^{1} \cdot w
   \rangle + C_{2} \|w\|^{2}.
\]
Since $\| D_{2} \phi^{1}_{0} \cdot w \| = 0$, by Theorem
\ref{theorem_comparison} again, there exists a $C_{3} > 0$ which is
independent of $r$ such that
\begin{equation}\label{lemma_inclination_5}
   \| D_{2} \phi^{1}_{t} \cdot w \| \leq \left[ \frac{C_{2}}{C_{3}} (e^{C_{3} T} - 1)
   \right]^{\frac{1}{2}} \|w\|.
\end{equation}

By (\ref{lemma_inclination_2}), (\ref{lemma_inclination_4}) and
(\ref{lemma_inclination_5}), there exist $K_{1} > 0$ and $K_{2} >
0$, which are independent of $r$, such that
\begin{equation}\label{lemma_inclination_6}
  \| D_{1} \phi^{1}_{t} \cdot w \| \geq K_{1} \|w\|,  \qquad \text{and} \quad  \| D_{2} \phi^{1}_{t} \cdot w
  \| \leq K_{2} \|w\|.
\end{equation}

Suppose $v=(v_{1}, v_{2}) \in H_{1} \oplus H_{2}$, and its
inclination is $\lambda_{0} = \frac{\|v_{2}\|}{\|v_{1}\|}$. By
(\ref{lemma_inclination_3}) and (\ref{lemma_inclination_6}), we have
the inclination of $D \phi_{t} \cdot v$ is
\begin{eqnarray*}
   \lambda_{1}  & = &  \frac{\| D_{2} \phi^{2}_{t} \cdot v_{2} \|}
   { \| D_{1} \phi^{1}_{t} \cdot v_{1} + D_{2} \phi^{1}_{t} \cdot v_{2} \| }
   \leq \frac{\| D_{2} \phi^{2}_{t} \cdot v_{2} \|}
   { \| D_{1} \phi^{1}_{t} \cdot v_{1} \| - \| D_{2} \phi^{1}_{t} \cdot v_{2} \| } \\
   & \leq &  \frac{ \|v_{2}\| }{ K_{1} \|v_{1}\| - K_{2} \|v_{2}\| }
   = \frac{\lambda_{0}}{K_{1} - K_{2} \lambda_{0}}.
\end{eqnarray*}
Thus $\lambda_{1}$ tends to $0$ when $\lambda_{0}$ tends to $0$.

We know that $\phi_{t}(x)$ enters $E(r)$ at most twice, and $K_{1}$ and
$K_{2}$ are independent of $r$. Thus, for an $\epsilon$ in the statement of this lemma, we can find the desired $\delta$ which is independent of $r$.
\end{proof}

\begin{remark}
Our Lemma \ref{lemma_inclination} is similar to the classical $\lambda$-Lemma (\cite[chap.\ 2, lem.\ 7.1 \& 7.2]{palis_de}) as both are to control the inclinations of tangent vectors. However, there is one remarkable difference between them: Lemma \ref{lemma_inclination} deals with a family of vector fields (\ref{equation_x_r}) parameterized by $r$ and the desired $\delta$ is independent of $r$.
\end{remark}

The following definition of \textit{filtration} is a special case of
that in hyperbolic dynamical systems (see \cite[p.\
1029]{nitecki_shub}).

\begin{definition}
A compact submanifold $M_{1}$ with boundary inside $M$ is a
filtration for $X$ if $\dim (M_{1}) = \dim (M)$,
$\phi_{t}(M_{1}) \subseteq \mathrm{Int} M_{1}$ for $t>0$, and $X$ is
transverse to $\partial M_{1}$. Here $\mathrm{Int} M_{1}$ is the
interior of $M_{1}$, and $\phi_{t}$ is the flow generated by $X$.
\end{definition}

\begin{lemma}\label{lemma_filtration}
Suppose $X$ satisfies transversality. If $p$ and $q$ are critical
points such that $p \npreceq q$, then there exists a filtration
$M_{1}$ for $X$ such that $p \in M-M_{1}$ and $q \in \mathrm{Int} M_{1}$.
\end{lemma}

Lemma \ref{lemma_filtration} can be proved as follows. The
transversality implies $``\preceq"$ is a partial order. We have $p
\npreceq q_{1}$ if $q_{1} \preceq q$. Using \cite[thm.\
4.1]{milnor2} repeatedly, we can modify $f$ to be a Morse function
$g$ such that $X$ is a negative gradient-like field for $g$ and
$g(q) < g(p)$. The proof is finished.

By Definition \ref{definition_gradient_like}, we have the following
obvious lemma.
\begin{lemma}\label{lemma_combine_field}
Suppose $X_{1}$ and $X_{2}$ are negative gradient-like fields of
$f$. Suppose $\sigma_{1}(x)$ and $\sigma_{2}(x)$ are nonnegative
smooth functions on $M$ such that $\sigma_{1} + \sigma_{2} > 0$.
Then $\sigma_{1} X_{1} + \sigma_{2} X_{2}$ is also a negative
gradient-like field for $f$.
\end{lemma}

Let $p$ be a critical point. Suppose there exists a Morse chart near
$p$ (see (\ref{morse_chart})), and $X(x_{1}, x_{2}) = (A x_{1}, -B
x_{2})$, where $A$ and $B$ are symmetric positive definite linear
operators. Similarly to Lemma \ref{lemma_inclination}, define
\[
  \mathcal{Y}_{r} (x_{1}, x_{2}) = (\rho_{r}(x_{1}) \rho_{r}(x_{2}) x_{1} +
  [1- \rho_{r}(x_{1}) \rho_{r}(x_{2})] A x_{1}, - B x_{2})
\]
in this Morse chart and $\mathcal{Y}_{r} = X$ out of this Morse
chart. For $s \in [0,1]$, define
\[
  \mathcal{Y}_{r,s} = (1-s) X + s \mathcal{Y}_{r}.
\]

By Lemma \ref{lemma_combine_field}, for all $s \in [0,1]$,
$\mathcal{Y}_{r,s}$ is a negative gradient-like field for $f$.

\begin{lemma}\label{lemma_partial_transversal}
Suppose in addition $X$ satisfies transversality. Then when $r$ is small enough,
we have the following conclusion.

Suppose $q_{1}$ and $q_{2}$ are two critical points which are not of
the following two cases: (1) $q_{2} \prec p \prec q_{1}$; or (2)
$q_{1} \prec p \prec q_{2}$. Then we have that $q_{1}$ and $q_{2}$
are transversal with respect to $\mathcal{Y}_{r,s}$ for all $s \in
[0,1]$. Here $``\prec"$ is defined with respect to $X$.
\end{lemma}
\begin{proof}
Clearly, $\mathcal{Y}_{r,s}$ differs from $X$ only in a neighborhood
$U_{r}$ of $p$. When $r$ tends to $0$, $U_{r}$ shrinks to $p$.

We may assume that $f(q) \neq f(p)$ for any critical point $q$ such
that $q \neq p$. If this is not true, perturb $f$ to be a Morse
function $\tilde{f}$ such that $X$ is a negative gradient-like field
for $\tilde{f}$, and $\tilde{f}(x) = f(x) + C$ in a neighborhood $U$
of $p$. Let $r$ be small enough such that $U_{r} \subseteq U$. Then
$\mathcal{Y}_{r,s}$ is also a negative gradient-like field for
$\tilde{f}$. For the rest of the proof we make the above assumption.

Suppose $U_{r} \subseteq M^{a,b}$ and $p$ is the unique singularity
in $M^{a,b}$. As in Definition \ref{definition_invariant_manifold},
we use notation $\mathcal{D}(*; *)$ and $\mathcal{A}(*; *)$ to
indicate the vector fields.

It's easy to see that $\mathcal{D}(p;\mathcal{Y}_{r,s}) \cap U_{r} =
\mathcal{D}(p;X) \cap U_{r}$. Since $\mathcal{Y}_{r,s}$ is identical to $X$ outside $U_{r}$, we have $\mathcal{D}(p;\mathcal{Y}_{r,s}) =
\mathcal{D}(p;X)$. Suppose that $q \in M^{a}$. Since
$\mathcal{Y}_{r,s}$ is identical to $X$ in $M-M^{a,b}$, we have
$\mathcal{A}(q;\mathcal{Y}_{r,s}) \cap M^{a} = \mathcal{A}(q;X) \cap
M^{a}$. Since $X$ satisfies transversality, we infer that $p$ and
$q$ are transversal in $M^{a}$ with respect to $\mathcal{Y}_{r,s}$.
By Lemma \ref{lemma_local_transversal}, $p$ and $q$ are transversal
globally. Similarly, if $q \in M-M^{a}$, $p$ and $q$ are also
transversal. As a result, $p$ and $q$ are transversal.

By the discussion above, it remains to check the case that $q_{1} \neq p$ and $q_{2} \neq p$.

If $p \nprec q$, by Lemma \ref{lemma_filtration}, there exists a
filtration $M_{1}$ such that $q \in \mathrm{Int}M_{1}$ and $p \in
M-M_{1}$. Let $r$ be small enough such that $U_{r} \subseteq
M-M_{1}$, then $\mathcal{Y}_{r,s}$ is identical to $X$ on $M_{1}$.
So $\mathcal{D}(q;\mathcal{Y}_{r,s}) = \mathcal{D}(q;X)$. Similarly,
if $q \nprec p$, we can get $\mathcal{A}(q;\mathcal{Y}_{r,s}) =
\mathcal{A}(q;X)$ when $r$ is small enough. Thus there exists $r_{0}
> 0$ such that the following holds. When $r < r_{0}$, we have, for
all $s \in [0,1]$, $\mathcal{D}(q;\mathcal{Y}_{r,s}) =
\mathcal{D}(q;X)$ if $p \nprec q$, and
$\mathcal{A}(q;\mathcal{Y}_{r,s}) = \mathcal{A}(q;X)$ if $q \nprec
p$.

In order to complete this proof, we only need to check the following
three cases.

(1). Case 1: $q_{1}$ and $q_{2}$ are in $M^{a}$.

Since $\mathcal{Y}_{r,s}$ is identical to $X$ on $M^{a}$ and $X$
satisfies transversality, we have $q_{1}$ and $q_{2}$ are
transversal in $M^{a}$. By Lemma \ref{lemma_local_transversal}, they
are transversal globally with respect to $\mathcal{Y}_{r,s}$.

(2). Case 2: $q_{1}$ and $q_{2}$ are in $M-M^{a}$.

By our assumption that $p$ is the unique critical point in $M^{a,b}$, we infer that $q_{1}$ and $q_{2}$ are actually in $M-M^{b}$. Similarly to Case (1), this case is also true.

(3). Case 3: one of $q_{1}$ and $q_{2}$ is in $M-M^{a}$ and the
other one is in $M^{a}$.

We may presume $q_{1} \in M-M^{a}$ and $q_{2} \in M^{a}$. Then actually $q_{1} \in M-M^{b}$. By the
assumption of this lemma, we have either $p \nprec q_{1}$ or $q_{2}
\nprec p$. Suppose $p \nprec q_{1}$. We have
$\mathcal{D}(q_{1};\mathcal{Y}_{r,s}) = \mathcal{D}(q_{1};X)$. Since $\mathcal{Y}_{r,s} = X$ on $M^{a}$ and
$X$ satisfies transversality, we have $q_{1}$ and $q_{2}$ are
transversal in $M^{a}$ with respect to $\mathcal{Y}_{r,s}$. By Lemma
\ref{lemma_local_transversal}, they are transversal globally.
Similarly, if $q_{2} \nprec p$, this is also true. Thus Case 3 is
also verified.

Since there are only finitely many critical points, we can find $r_{0} >0$ such that all $r < r_{0}$ satisfy the conclusion of this lemma.
\end{proof}

We shall strengthen Lemma \ref{lemma_partial_transversal} to get the
transversality of $\mathcal{Y}_{r,s}$. Recall a classical result on
transversality at first.

Suppose $U$ is a neighborhood of $p$ such that $U$ is identified
with a neighborhood of $0$ in $T_{p} M = H_{1} \oplus H_{2}$, and
$p$ is identified with $0$, where $H_{1} = T_{p} \mathcal{D}(p;X)$
and $H_{2} = T_{p} \mathcal{A}(p;X)$. Furthermore, suppose
$\mathcal{D}(p;X) \cap U \subseteq H_{1}$ and $\mathcal{A}(p;X) \cap
U \subseteq H_{2}$. Then we have the following crucial fact: When
$U$ is small enough, there exists $\Lambda > 0$ such that for any $q_{1}
\succeq p$ and any $x \in \mathcal{D}(q_{1};X) \cap U$, there exists
a linear space $V_{x}^{d} \subseteq T_{x} \mathcal{D}(q_{1};X)$ such
that $\dim (V_{x}^{d}) = \dim (H_{1})$ and the inclination
of $V_{x}^{d}$ with respect to $H_{1}$ is less than $\Lambda$.
Similarly, for any $q_{2} \preceq p$ and any $x \in
\mathcal{A}(q_{2};X) \cap U$, there exists $V_{x}^{a} \subseteq
T_{x} \mathcal{A}(q_{2};X)$ such that $\dim (V_{x}^{a}) =
\dim (H_{2})$ and the inclination of $V_{x}^{a}$ with respect
to $H_{2}$ is also less than $\Lambda$. In addition, $\Lambda$ tends
to $0$ when $U$ shrinks to $p$. This fact follows from the
transversality of $X$ and the estimate of the $\lambda$-Lemma.
(Note: the $\lambda$-Lemma is also named the Inclination Lemma.)
On the contrary, we assume this fact holds but do not assume the
transversality of $X$. If $\Lambda < 1$, then, for any $x \in
\mathcal{D}(q_{1};X) \cap \mathcal{A}(q_{2};X) \cap U$, we have
\[
  T_{x} M = H_{1} \oplus H_{2} = V_{x}^{d} \oplus V_{x}^{a} = T_{x}
  \mathcal{D}(q_{1};X) + T_{x} \mathcal{A}(q_{2};X).
\]
So we infer that $\mathcal{D}(q_{1};X)$ and $\mathcal{A}(q_{2};X)$
are transversal in $U$. The above argument is the key part of the
proof of that, for Morse-Smale dynamical systems, transversality is
preserved under small $C^{1}$ perturbations. All of these are
addressed in \cite[lem.\ 1.11 and thm.\ 3.5]{palis}. In the proof of
the following lemma, we shall apply a similar argument to large
$C^{1}$ perturbations of $X$.

\begin{lemma}\label{lemma_transversl}
Suppose $X$ in addition satisfies transversality. When $r$ is small enough, we
have $\mathcal{Y}_{r,s}$ satisfies transversality for all $s \in
[0,1]$.
\end{lemma}
\begin{proof}
By Lemma \ref{lemma_partial_transversal}, it suffices to prove that
$\mathcal{D}(q_{1}; \mathcal{Y}_{r,s})$ is transverse to
$\mathcal{A}(q_{2}; \mathcal{Y}_{r,s})$ if $q_{2} \prec p \prec
q_{1}$.

Similarly to the proof of Lemma \ref{lemma_partial_transversal}, we
assume that $p$ is the unique critical point in
$M^{f(p)-\epsilon,f(p)+\epsilon}$. Let $U$ be the neighborhood of
$p$ in the argument before this lemma. Let $D$ be an open subset of
$f^{-1}(f(p)+\epsilon) \cap U$ such that $D \supseteq
f^{-1}(f(p)+\epsilon) \cap \mathcal{A}(p;X)$. Let $U_{0} =
[\phi([0,+\infty),D) \cup \mathcal{D}(p;X)] \cap
M^{f(p)-\epsilon,f(p)+\epsilon}$. Then $U_{0}$ is a neighborhood of
$p$ and is relatively open in $M^{f(p)-\epsilon,f(p)+\epsilon}$.
When $\epsilon$ tends to $0$ and $D$ shrinks, $U_{0}$ shrinks to
$p$. (In Figure \ref{figure_neighborhood}, the shadowed part is
$U_{0}$, the arrows indicate the the directions of the flows.)
Denote the flow generated by $\mathcal{Y}_{r,s}$ by
$\phi^{r,s}_{t}$.

\begin{figure}[!htbp]
\centering
\includegraphics[scale=0.4]{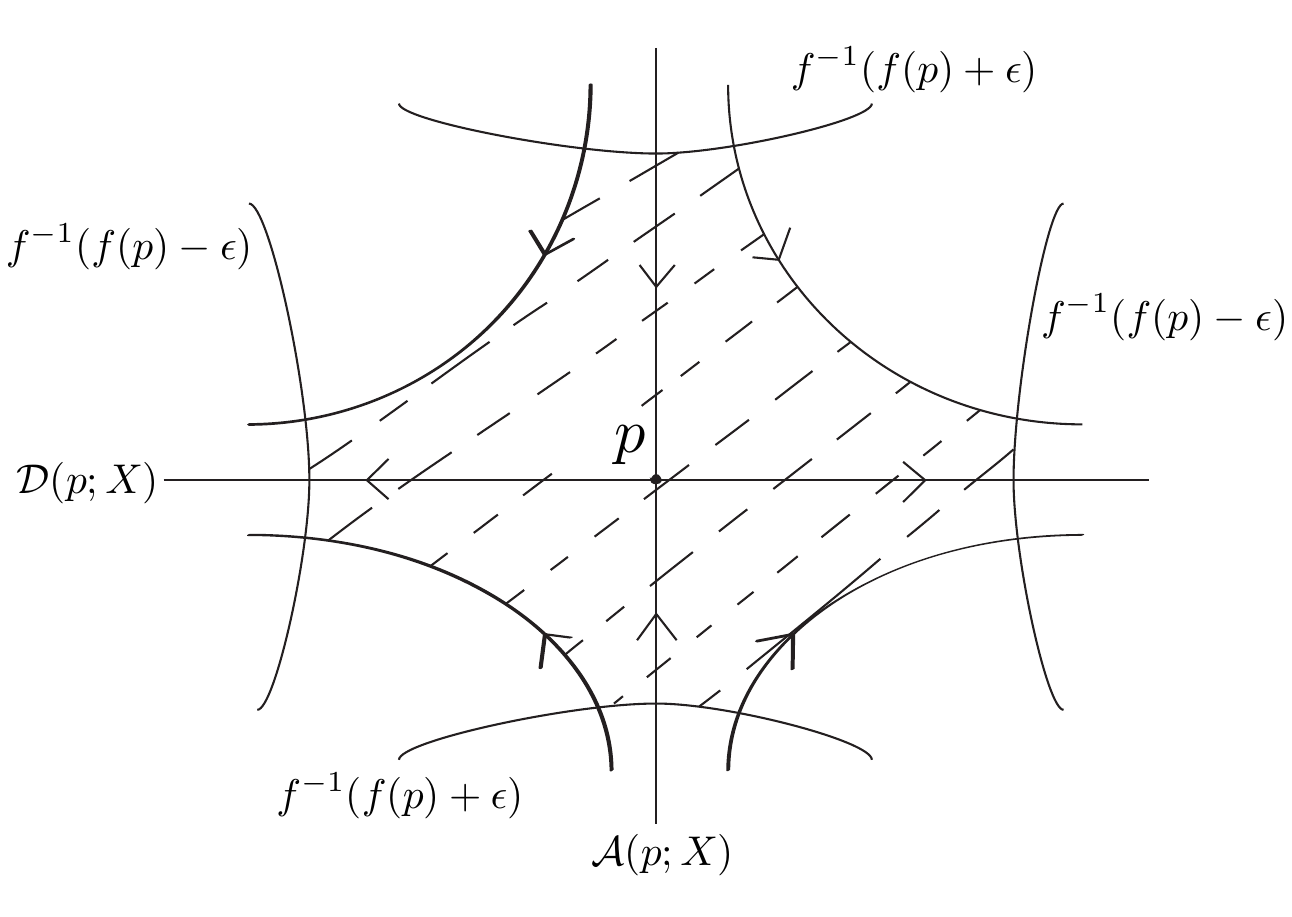} \caption{Neighborhood $U_{0}$}
\label{figure_neighborhood}
\end{figure}

By the construction of $U_{0}$, we know that, for any $z \in M^{f(p) - \epsilon, f(p) + \epsilon} - U_{0}$, the orbit of $\phi_{t}(x)$ has no intersection with $U_{0}$. Similarly, we can construct $U_{1} \subseteq U_{0}$ such that $U_{1}$ is a closed neighborhood of $p$, and, for any $z \in M^{f(p) - \epsilon, f(p) + \epsilon} - U_{1}$, the orbit of $\phi_{t}(x)$ has no intersection with $U_{1}$. Let $r$ be small enough such that $\mathcal{Y}_{r,s}$ is identical to $X$ out of $U_{1}$. We infer that, for any $z \in M^{f(p) - \epsilon, f(p) + \epsilon} - U_{i}$, the orbit of $\phi_{t}^{r,s}(z)$ also has no intersection with $U_{i}$, and $\phi_{t}^{r,s}(z) = \phi_{t}(z)$, where $i=0,1$. Then, for any $x \in f^{-1}(f(p) + \epsilon) \cap U_{0} - \mathcal{A}(p;X)$, we have $\phi_{t}^{r,s}(x) \in f^{-1}(f(p) - \epsilon)$ for some $t>0$ and $\phi^{r,s}([0,t],x) \subset U_{0}$. Therefore, Lemma \ref{lemma_inclination} is applicable to the flow segment $\phi^{r,s}([0,t],x)$. Now we apply that Lemma.

We know that
\[
  \mathcal{Y}_{r,s}(x_{1},x_{2}) = (\rho_{r}(x_{1}) \rho_{r}(x_{2}) (sI + (1-s)A) x_{1} +
  [1- \rho_{r}(x_{1}) \rho_{r}(x_{2})] A x_{1}, - B x_{2}),
\]
and there exist $\alpha_{0} > 0$, $\alpha_{1} > 0$ and $\beta > 0$
such that, for any $s \in [0,1]$, we have
\[
  \alpha_{0} I \leq sI + (1-s)A \leq \alpha_{1} I, \qquad \alpha_{0} I \leq A \leq \alpha_{1}
  I, \qquad \text{and} \quad   \beta I \leq B.
\]
By Lemma \ref{lemma_inclination}, there exists $\delta > 0$ such
that the following holds. Suppose $x \in \mathcal{D}(q_{1};X) \cap
f^{-1}(f(p)+\epsilon) \cap U_{0}$, and $V_{x}^{d} \subseteq T_{x}
\mathcal{D}(q_{1};X)$ is the space described before this lemma. If
the inclination of $V_{x}^{d}$ with respect to $H_{1}$ is less than
$\delta$, then, in $U_{0}$, the inclination of
$D \phi^{r,s}_{t}(V_{x}^{d})$ with respect to $H_{1}$ is less than
$1$. It's necessary to point out that $\delta$ is independent of $r$
and $s$.

Clearly, $\mathcal{D}(q_{1};X) \cap f^{-1}([f(p) + \epsilon,
+\infty)) = \mathcal{D}(q_{1};\mathcal{Y}_{r,s}) \cap f^{-1}([f(p) +
\epsilon, +\infty))$ and $\mathcal{A}(q_{2}; X) \cap
M^{f(p)-\epsilon} = \mathcal{A}(q_{2}; \mathcal{Y}_{r,s}) \cap
M^{f(p)-\epsilon}$. Since $X$ satisfies transversality, by the
argument before this lemma, we can choose $U_{0}$ be small enough
such that the following holds. For any $x \in \mathcal{D}(q_{1};X)
\cap f^{-1}(f(p)+\epsilon) \cap U_{0}$, the inclination of
$V_{x}^{d}$ with respect to $H_{1}$ is less than $\delta$, and, for
any $y \in \mathcal{A}(q_{2};X) \cap f^{-1}(f(p)-\epsilon) \cap
U_{0}$, the inclination of $V_{y}^{a}$ with respect to $H_{2}$ is
less than $1$. Here $V_{x}^{d} \subseteq T_{x} \mathcal{D}(q_{1};X)
= T_{x} \mathcal{D}(q_{1};\mathcal{Y}_{r,s})$ and $V_{y}^{a}
\subseteq T_{y} \mathcal{A}(q_{2}; X) = T_{y} \mathcal{A}(q_{2};
\mathcal{Y}_{r,s})$. Thus, if $\phi_{t}^{r,s}(x) = y$, then the
inclination of $V_{y}^{d} = D \phi_{t}^{r,s} \cdot V_{x}^{d}$ with
respect to $H_{1}$ is less than $1$. Here $V_{y}^{d} \subseteq T_{y}
\mathcal{D}(q_{1};\mathcal{Y}_{r,s})$. By the argument before this
lemma again, we have $T_{y} M = V_{y}^{d} \oplus V_{y}^{a}$. So
$\mathcal{D}(q_{1};\mathcal{Y}_{r,s})$ and $\mathcal{A}(q_{2};
\mathcal{Y}_{r,s})$ are transversal in $f^{-1}(f(p)-\epsilon) \cap
U_{0}$.

Furthermore, $\mathcal{D}(q_{1};X) \cap
(M^{f(p)-\epsilon,f(p)+\epsilon} - U_{1})=
\mathcal{D}(q_{1};\mathcal{Y}_{r,s}) \cap
(M^{f(p)-\epsilon,f(p)+\epsilon} - U_{1})$ and $\mathcal{A}(q_{2};X)
\cap (M^{f(p)-\epsilon,f(p)+\epsilon} - U_{1})=
\mathcal{A}(q_{2};\mathcal{Y}_{r,s}) \cap
(M^{f(p)-\epsilon,f(p)+\epsilon} - U_{1})$. Thus
$\mathcal{D}(q_{1};\mathcal{Y}_{r,s})$ and
$\mathcal{A}(q_{2};\mathcal{Y}_{r,s})$ are transversal in
$M^{f(p)-\epsilon,f(p)+\epsilon} - U_{1}$.

In summary, $\mathcal{D}(q_{1};\mathcal{Y}_{r,s})$ and
$\mathcal{A}(q_{2};\mathcal{Y}_{r,s})$ are transversal in
$f^{-1}(f(p) - \epsilon)$. By Lemma \ref{lemma_local_transversal},
they are transversal globally.

Since there are only finitely many critical points, we can find $r_{0} >0$ such that all $r < r_{0}$ satisfy the conclusion of this lemma.
\end{proof}

\begin{proof}[Proof of Theorem \ref{theorem_regular_path}]
First, we construct the regular path. Since the number of critical points is finite, it suffices to prove that, for
any critical point $p$, we can construct a regular path
$\mathcal{Y}$ such that $\mathcal{Y}_{0} = X$ and $\mathcal{Y}_{1}$
is locally trivial at $p$.

By Theorem \ref{theorem_Morse}, there exists a coordinate chart $U$
near $p$ such that $p$ has coordinate $(0,0)$,
\begin{equation}\label{theorem_regular_path_2}
  f(x_{1}, x_{2}) = f(p) - \frac{1}{2} \langle x_{1}, x_{1} \rangle + \frac{1}{2} \langle
  x_{2}, x_{2} \rangle,
\end{equation}
$\mathcal{D}(p;X) \cap U = \{(x_{1},0)\}$ and $\mathcal{A}(p;X) \cap
U = \{(0,x_{2})\}$. We immediately see that $DX(p)$ is a diagonal
\begin{equation}\label{theorem_regular_path_1}
DX(p) =
\begin{bmatrix}
A & 0 \\
0 & -B
\end{bmatrix},
\end{equation}
where $A$ and $B$ are symmetric and positive definite. By (\ref{theorem_regular_path_1}), the vector field $(A x_{1}, -B
x_{2})$ is the linearization of $X$ at $p$. By (\ref{theorem_regular_path_2}), it is also negative gradient-like for $f$ near $p$.

Let $\rho_{r}$ be the bump function defined before. For convenience,
for all $x = (x_{1}, x_{2})$, let $\rho_{r}(x)$ denote $\rho_{r}(\|x\|)$. Let $R(x) = X(x) - (A x_{1}, -B x_{2})$. Then we have
$\| \rho_{r}(x) R(x) \|$ and $\| D [\rho_{r}(x) R(x)] \|$ tend to
$0$ when $r$ tends to $0$. Since the transversality of $X$ is
preserved under small $C^{1}$ perturbations, we have
$\mathcal{Z}_{s} = X - s \rho_{r} R$ is a regular path when $r$ is
small enough and $s \in [0,1]$. By Lemma \ref{lemma_combine_field}, $\mathcal{Z}_{1}$ and then each $\mathcal{Z}_{s} = (1-s) X + s \mathcal{Z}_{1}$ are negative gradient-like for $f$. Since $\mathcal{Z}_{1}(x) = (A
x_{1}, -B x_{2})$ near $p$, by Lemma \ref{lemma_transversl}, we can
construct a regular path $\mathcal{Z}_{[1,2]}$ such that
$\mathcal{Z}_{2}(x) = (x_{1}, -B x_{2})$ near $p$. Since
$-\mathcal{Z}_{2}$ is a negative gradient-like field for $-f$, using
Lemma \ref{lemma_transversl} again, we can construct a regular path
$-\mathcal{Z}_{[2,3]}$ for $-f$ such that $-\mathcal{Z}_{3}(x) = (-x_{1},
x_{2})$ near $p$. We get the desired path by defining
$\mathcal{Y}_{s} = \mathcal{Z}_{3s}$.

Second, we prove the existence of the conjugacy $h$.

By the proof in \cite[thm.\ 5.2]{palis_smale}, we know that, for
each $\mathcal{Y}_{s_{0}}$, there is a topological equivalence
$h_{s_{0}}$ between $\mathcal{Y}_{s_{0}}$ and $\mathcal{Y}_{s}$ such
that $h_{s_{0}}(p) = p$ for all critical points $p$ when $s$ is
close to $s_{0}$ enough. In addition, since the flow generated by
$\mathcal{Y}_{s_{0}}$ has no closed orbits, by the comment in
\cite[p.\ 231]{palis_smale}, we know that $h_{s_{0}}$ is actually a
conjugacy. Thus it's easy to get the desired conjugacy $h$.
\end{proof}

\begin{remark}\label{remark_morse_path}
In order to guarantee that the path $\mathcal{Y}$ in Theorem \ref{theorem_regular_path} is negative gradient-like for $f$, we need to find a Morse chart satisfying both (\ref{theorem_regular_path_2}) and (\ref{theorem_regular_path_1}). Theorem \ref{theorem_Morse} trivially yields this chart. (Actually, Theorem \ref{theorem_Morse} provides us more than what we actually need.) It's necessary to point out that a general Morse chart does \textit{not} necessarily satisfy (\ref{theorem_regular_path_1}) because $DX(p)$ depends on the metric. Thus the usual Morse Lemma is not sufficient for us. This special Morse chart may be constructed without using Theorem \ref{theorem_Morse}. Nevertheless, we present this theorem because it may be of independent interest.
\end{remark}

\begin{remark}\label{remark_path}
The regular path in \cite[lem.\ 2]{newhouse_peixoto} consists of the
Morse-Smale vector fields without closed orbits. In this case, $D
X(p) = \text{diag} (A, -B)$ for singularities $p$, where $A$ and $B$ are linear
isomorphisms whose eigenvalues have positive real parts. The paper
\cite{newhouse_peixoto} claims that there exists a regular path
connecting $X$ with $Y$ such that $Y(x_{1}, x_{2}) = (2x_{1}, -
2x_{2})$ near each singularity. Thus, in the setting of dynamical
systems, this result is more general than Theorem
\ref{theorem_regular_path}. However, Theorem
\ref{theorem_regular_path} has the advantage that its vector fields
are negative gradient-like for $f$. Furthermore, the argument in this paper
can also be used to verify the result in \cite{newhouse_peixoto}.
This is because we can choose a metric near each critical point, for
example, by the real Jordan canonical form, such that the above
operators $A$ and $B$ satisfy (\ref{first_operator}) and
(\ref{second_operator}).
\end{remark}

\section{A Reduction Lemma}\label{section_reduction_lemma}
In this paper, we shall prove theorems for noncompact manifolds with
proper Morse functions. However, the manifold in Theorem
\ref{theorem_regular_path} is required to be compact. (As we have seen, the proof of Theorem \ref{theorem_regular_path} heavily relies on the compactness of the manifold.) The following
lemma reduces the proper case to the compact case.

\begin{lemma}\label{lemma_reduction}
Suppose $M$ is a compact manifold with boundary $\partial M = M_{1}
\sqcup M_{2}$. Here $M_{i}$ ($i=1,2$) may be empty. Suppose $f$ is a
Morse function on $M$ such that $f|_{M_{1}} \equiv a$, $f|_{M_{2}}
\equiv b$, $a$ and $b$ are regular values of $f$, and $a < b$.
Suppose $X$ is a negative gradient-like vector field for $f$, and
$X$ satisfies transversality. Then there exist a compact manifold
$\widetilde{M}$ without boundary and a smooth embedding $i: M
\hookrightarrow \widetilde{M}$ such that the following holds. There
exist a Morse function $\widetilde{f}$ and its negative
gradient-like vector field $\widetilde{X}$ on $\widetilde{M}$.
They are extensions of $f$ and $X$ respectively, and $\widetilde{X}$
satisfies transversality. For any critical points $p$ and $q$ in
$M$, we have $\mathcal{D}(p; \widetilde{X}) \cap \mathcal{A}(q;
\widetilde{X}) = \mathcal{D}(p;X) \cap \mathcal{A}(q;X)$.
Furthermore, $\mathcal{D}(p; \widetilde{X}) = \mathcal{D}(p;X)$ and
$\widetilde{f}|_{\widetilde{M} - M} > b$ if $M_{1} = \emptyset$; and
$\mathcal{A}(p; \widetilde{X}) = \mathcal{A}(p;X)$ and
$\widetilde{f}|_{\widetilde{M} - M} < a$ if $M_{2} = \emptyset$.
\end{lemma}

The proof of Lemma \ref{lemma_reduction} is based on Milnor's sliding invariant (descending or
ascending) manifolds in \cite[thm.\ 5.2]{milnor2}. Suppose $g$ is a Morse function on a compact manifold and $\xi$ is a negative gradient-like field for $g$. Basically, there are two ways of modifying $\xi$ to get
transversality. Method 1 is sliding the descending manifolds one by
one with the order from critical points with lower values to those
with higher values. In this case, one repeats using \cite[thm.\ 5.2]{milnor2} to slide each $\mathcal{D}(p)$ such that $\mathcal{D}(p)$ becomes transverse to all $\mathcal{A}(q)$ if $g(q) < g(p)$. On the contrary, Method 2 is sliding the
ascending manifolds one by one with the order from critical points
with higher values to those with lower values. A key point is that one only needs to change the filed $\xi$ in an arbitrarily small neighborhood of $p$ when sliding the invariant manifolds of $p$. Our method is a
combination of the above two methods.

\begin{proof}[Proof of Lemma \ref{lemma_reduction}]
If $\partial M = \emptyset$, let $\widetilde{M} = M$, the proof is
finished. Now we assume $\partial M \neq \emptyset$.

Let $\widetilde{M}$ be the double of $M$. Extend $f$ to be a Morse function
$\widetilde{f}$ such that $a$ and $b$ are regular values of $\widetilde{f}$, and
extend $X$ to be $\widetilde{X}$ which is a negative gradient-like
field for $\widetilde{f}$. (Figure \ref{figure_tilde_m} illustrates
the manifold $\widetilde{M}$, where the Morse function is the height
function and the shadowed part is $M$.) We shall modify
$\widetilde{X}$ such that it satisfies transversality.

\begin{figure}[!htbp]
\centering
\includegraphics[scale=0.4]{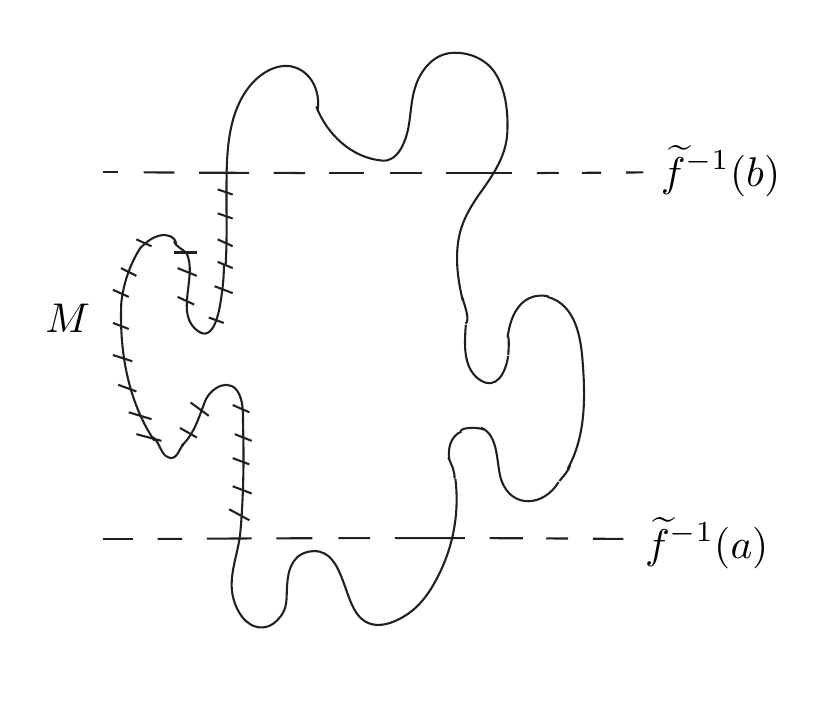} \caption{Manifold $\widetilde{M}$}
\label{figure_tilde_m}
\end{figure}

In this proof, we say two critical points $\widetilde{p}$ and
$\widetilde{q}$ of $\widetilde{f}$ are transversal if they are
transversal with respect to $\widetilde{X}$.

Step 1: we show the transversality between $p \in M$ and $q \in M$.
Since $\mathcal{D}(p; \widetilde{X}) \subseteq M \cup \mathrm{Int}
\widetilde{M}^{a}$, we have $\mathcal{D}(p; \widetilde{X}) \cap
\widetilde{M}^{a,b} = \mathcal{D}(p; \widetilde{X}) \cap M =
\mathcal{D}(p; X)$. Similarly, $\mathcal{A}(p; \widetilde{X}) \cap
\widetilde{M}^{a,b} = \mathcal{A}(p; X)$. Since $X$ satisfies
transversality, $p$ and $q$ are transversal in
$\widetilde{M}^{a,b}$. By Lemma \ref{lemma_local_transversal}, they
are transversal globally. This shows the transversality between $p$
and $q$ does not depend on the extension of $X$. So, no matter how
$\widetilde{X}$ is changed outside of $M$, $p$ and $q$ are always
transversal if they are in $M$.

Step 2: we modify $\widetilde{X}$ in $\widetilde{M}^{a}$. We make
modifications near each critical point $\widetilde{p}$ in
$\widetilde{M}^{a}$ with the order from critical points with higher
values to those with lower values. Slide $\mathcal{A}(\widetilde{p};
\widetilde{X})$ for each $\widetilde{p} \in \widetilde{M}^{a}$ such
that $\widetilde{p}$ is transverse to each $\widetilde{q} \in M \cup
\widetilde{M}^{a}$ with $\widetilde{f}(\widetilde{q}) >
\widetilde{f}(\widetilde{p})$. (Here, for all $\widetilde{q} \in M$,
we have $\widetilde{f}(\widetilde{q}) >
\widetilde{f}(\widetilde{p})$.) Thus, for all $\widetilde{p}$ and
$\widetilde{q}$ in $M \cup \widetilde{M}^{a}$, they are transversal
globally after these modifications. By Lemma
\ref{lemma_local_transversal} and Step 1, no matter how
$\widetilde{X}$ is changed outside of $M \cup \widetilde{M}^{a}$,
$\widetilde{p}$ and $\widetilde{q}$ are still transversal globally
because they are still transversal in $\widetilde{M}^{a}$.

Step 3: we modify $\widetilde{X}$ in $\widetilde{M}^{b} - [M \cup
\widetilde{M}^{a}]$. To do this, we slide
the descending manifolds with the order from critical points with
lower values to those with higher values. More precisely, slide
$\mathcal{D}(\widetilde{p}; \widetilde{X})$ for each $\widetilde{p}
\in \widetilde{M}^{b} - [M \cup \widetilde{M}^{a}]$ such that
$\widetilde{p}$ is transverse to all $\widetilde{q} \in
\widetilde{M}^{b} - M$ with $\widetilde{f}(\widetilde{q}) <
\widetilde{f}(\widetilde{p})$. (Here, for all $\widetilde{q} \in
\widetilde{M}^{a}$, we have $\widetilde{f}(\widetilde{q}) <
\widetilde{f}(\widetilde{p})$.) We claim that, for all
$\widetilde{p}$ and $\widetilde{q}$ in $\widetilde{M}^{b}$, they are
transversal. By what we have already proved, it remains to show that, for each $p \in M$ and
$\widetilde{q} \in \widetilde{M}^{a,b} - M$, we have $p$ and
$\widetilde{q}$ are transversal. Clearly,
$\mathcal{D}(\widetilde{q}; \widetilde{X}) \subseteq
\widetilde{M}^{b} - M$, thus $\mathcal{D}(\widetilde{q};
\widetilde{X}) \cap \widetilde{M}^{a,b} \subseteq
\widetilde{M}^{a,b} - M$. Since $\mathcal{A}(p; \widetilde{X}) \cap
\widetilde{M}^{a,b} \subseteq M$, we get $\mathcal{D}(\widetilde{q};
\widetilde{X}) \cap \mathcal{A}(p; \widetilde{X}) \cap
\widetilde{M}^{a,b} = \emptyset$. So $\mathcal{D}(\widetilde{q};
\widetilde{X}) \cap \mathcal{A}(p; \widetilde{X}) = \emptyset$.
Similarly, $\mathcal{A}(\widetilde{q}; \widetilde{X}) \cap
\mathcal{D}(p; \widetilde{X}) = \emptyset$. We infer that $p$ and
$\widetilde{q}$ are transversal. The above claim is proved. By Lemma
\ref{lemma_local_transversal} again, no matter how
$\widetilde{X}$ is changed outside of $\widetilde{M}^{b}$, all critical points in
$\widetilde{M}^{b}$ are still mutually transverse.

Step 4: we modify $\widetilde{X}$ on $\widetilde{M} -
\widetilde{M}^{b}$. Slide the descending manifolds with the order
from critical points with lower values to those with higher values.
We eventually get that $\widetilde{X}$ satisfies transversality.

By the above argument, for all $p$ and $q$ in $M$, we have
$\mathcal{D}(p; \widetilde{X}) \subseteq M \cup
\widetilde{f}^{-1}((-\infty,a))$, $\mathcal{A}(q; \widetilde{X})
\subseteq M \cup \widetilde{f}^{-1}((b,+\infty))$, $\mathcal{D}(p;
\widetilde{X}) \cap M = \mathcal{D}(p; X)$ and $\mathcal{A}(q;
\widetilde{X}) \cap M = \mathcal{A}(q; X)$. Thus
\[
  \mathcal{D}(p; \widetilde{X}) \cap \mathcal{A}(q;
  \widetilde{X}) =(\mathcal{D}(p; \widetilde{X}) \cap M) \cap (\mathcal{A}(q;
  \widetilde{X}) \cap M) = \mathcal{D}(p; X) \cap \mathcal{A}(q;
  X).
\]

Suppose $M_{1} = \emptyset$. Clearly, we can construct
$\widetilde{f}$ such that $\widetilde{f}|_{\widetilde{M} - M} > b$.
Thus, for any $p \in M$, we have $\mathcal{D}(p; \widetilde{X})
\subseteq M$ and $\mathcal{D}(p; \widetilde{X}) = \mathcal{D}(p;
X)$. Similarly, the conclusion is true in the case of $M_{2} =
\emptyset$.
\end{proof}

\section{Moduli Spaces and Topological
Equivalence}\label{section_moduli_spaces}
In this section, we shall review the definitions of moduli spaces
and their compactifications. These definitions are standard in the
literature (see e.g. \cite{latour}, \cite{burghelea_haller},
\cite{schwarz} and \cite{qin}). There are several ways to define the
topology of these spaces. The definitions in this paper follow those in \cite[thms. 3.3, 3.4
and 3.5]{qin}.

The paper \cite{qin} focuses on the negative \textit{gradient}
vector fields. This paper deals with the negative
\textit{gradient-like} vector fields. By Lemma \ref{lemma_gradient},
there is no difference.

After this review, we shall prove Theorem
\ref{theorem_topological_equivalence}. This theorem shows that
topologically equivalent negative gradient-like fields have
homeomorphic compactified moduli spaces. In other words, the
compactified moduli spaces are invariants of topological
equivalence. In this paper, the application of topological
equivalence to Morse theory is based on this theorem.

Let $M$ be a finite dimensional manifold. Let $f$ be a proper Morse
function on $M$ and $X$ be a negative gradient-like vector field for
$f$. Assume $X$ satisfies transversality. Denote by $\phi_{t}(x)$
the flow generated by $X$ with initial value $x$ and time $t$. Define an
equivalence relation on $M$ by
\[
  x \thicksim y \quad \Leftrightarrow \quad \text{$y = \phi_{t}(x)$ for some $t \in (-\infty,
  +\infty)$}.
\]
Then $x \thicksim y$ if and only if $x$ and $y$ lie on the same flow
line. Suppose $p$ and $q$ are critical points of $f$. Define
$\mathcal{W}(p,q) = \mathcal{D}(p) \cap \mathcal{A}(q)$. Then
$\mathcal{W}(p,q)$ is a smoothly embedded submanifold of $M$. Define
$\mathcal{M}(p,q) = \mathcal{W}(p,q) / \thicksim$. We define the
smooth structure of $\mathcal{M}(p,q)$ as follows. Choose a regular
value $a \in (f(q), f(p))$. Then each flow line in
$\mathcal{W}(p,q)$ intersects $f^{-1}(a)$ exactly at one point. This
identifies $\mathcal{M}(p,q)$ with $\mathcal{W}(p,q) \cap f^{-1}(a)$
naturally. We transfer the smooth structure of $\mathcal{W}(p,q)
\cap f^{-1}(a)$ to $\mathcal{M}(p,q)$ by this identification.
Clearly, this definition does not depend on the choice of $a$.
Furthermore, the natural projection from $\mathcal{W}(p,q)$ to
$\mathcal{M}(p,q)$ is a smooth submersion.

It's well known that $\dim (\mathcal{W}(p,q)) = \mathrm{ind}(p) -
\mathrm{ind}(q)$ and $\dim (\mathcal{M}(p,q)) = \mathrm{ind}(p) -
\mathrm{ind}(q) - 1$ if $p \succ q$.

We shall generalize the concept of flow lines. Suppose $\gamma$ is a
flow line. If it passes through a singularity, it is a constant flow
line. Otherwise, it is nonconstant. The following definitions follow
\cite[sec.\ 2]{qin}

\begin{definition}\label{generalized_flow_line}
An ordered sequence of flow lines $\Gamma = (\gamma_{1},\cdots,
\gamma_{n})$, $n \geq 1$, is a generalized flow line if
$\gamma_{i}(+\infty) = \gamma_{i+1}(-\infty)$ and $\gamma_{i}$ are
constant or nonconstant alternatively according to the order of
their places in the sequence. We call $\gamma_{i}$ a component of
$\Gamma$.
\end{definition}

\begin{definition}\label{generalized_flow_line_points}
Suppose $x$ and $y$ are two points in $M$. A generalized flow line
$(\gamma_{1},\cdots, \gamma_{n})$ connects $x$ with $y$ if there
exist $t_{1}, t_{2} \in (-\infty, +\infty)$ such that
$\gamma_{1}(t_{1}) = x$ and $\gamma_{n}(t_{2}) = y$. A point $z$ is
a point on $(\gamma_{1},\cdots, \gamma_{n})$ if there exists
$\gamma_{i}$ and $t \in (-\infty, +\infty)$ such that $\gamma_{i}(t)
= z$.
\end{definition}

As mentioned before, $``\succeq"$ is a partial order
because of transversality.
\begin{definition}\label{critical_sequence}
An ordered set $I = \{ r_{0}, r_{1}, \cdots, r_{k+1} \}$ is a
critical sequence if $r_{i}$ ($i=0, \cdots, k+1$) are critical
points and $r_{0} \succ r_{1} \succ \cdots \succ r_{k+1}$. We call
$r_{0}$ the head of $I$, and $r_{k+1}$ the tail of $I$. The length
of $I$ is $|I|=k$. In particular, if $I = \{ r_{0} \}$, then $|I|=-1$.
\end{definition}

Suppose $I = \{ r_{0}, r_{1}, \cdots, r_{k+1} \}$ is a critical
sequence. If $k+1 >0$, define
\[
  \mathcal{M}_{I} = \prod_{i=0}^{k} \mathcal{M}(r_{i}, r_{i+1}).
\]
On the contrary, if $I = \{ r_{0} \}$, then define $\mathcal{M}_{I}$ as the one point set $\{ \beta(r_{0}) \}$, where $\beta(r_{0})$ is the constant flow line passing through $r_{0}$.

For $p \succ q$, define a space $\overline{\mathcal{M}(p,q)}$ as
\begin{equation}\label{equation_m(p,q)}
\overline{\mathcal{M}(p,q)} = \bigsqcup_{I} \mathcal{M}_{I},
\end{equation}
where the disjoint union is over all critical sequence with head $p$
and tail $q$.

We can give $\overline{\mathcal{M}(p,q)}$ another equivalent
definition which is sometimes more convenient. If $\alpha \in
\mathcal{M}_{I} \subseteq \overline{\mathcal{M}(p,q)}$, then $\alpha
= (\gamma_{0}, \cdots, \gamma_{k})$, where $\gamma_{i} \in
\mathcal{M}(r_{i}, r_{i+1})$, $r_{0}=p$ and $r_{k+1}=q$. Let $\beta(r_{i})$ denote the
constant flow line passing through $r_{i}$. We can
identify $\alpha$ with the generalized flow line $(\beta(r_{0}),
\gamma_{0}, \beta(r_{1}), \cdots, \gamma_{k},$ $\beta(r_{k+1}))$
connecting $p$ with $q$. Thus we get
\[
\overline{\mathcal{M}(p,q)}  = \{ \Gamma \mid \Gamma \text{ is a
generalized flow line connecting $p$ with $q$} \}.
\]

Suppose the critical values of $f$ divide $[f(q), f(p)]$ into $l+1$
intervals $[c_{i+1}, c_{i}]$ ($i=0, \cdots, l$), where $c_{0} =
f(p)$ and $c_{l+1} = f(q)$. Choose a regular value $a_{i} \in
(c_{i+1}, c_{i})$. The generalized flow line $\Gamma \in
\overline{\mathcal{M}(p,q)}$ intersects with $f^{-1}(a_{i})$ at
exactly one point $x_{i}(\Gamma)$. There is an evaluation map $E:
\overline{\mathcal{M}(p,q)} \rightarrow \prod_{i=0}^{l}
f^{-1}(a_{i})$ which is injective and is defined as
\begin{equation}\label{evalution_m(p,q)}
E(\Gamma) = (x_{0}(\Gamma), \cdots, x_{l}(\Gamma)).
\end{equation}

\begin{definition}\label{definition_m(p,q)}
For $p \succ q$, define the set $\overline{\mathcal{M}(p,q)}$ as
(\ref{equation_m(p,q)}). Equip $\overline{\mathcal{M}(p,q)}$ with
the unique topology such that the evaluation map $E:
\overline{\mathcal{M}(p,q)} \rightarrow \prod_{i=0}^{l}
f^{-1}(a_{i})$ in (\ref{evalution_m(p,q)}) is a topological
embedding. We call $\overline{\mathcal{M}(p,q)}$ the compactified
moduli space of $\mathcal{M}(p,q)$.
\end{definition}

It's easy to see that the definition of the topology of
$\overline{\mathcal{M}(p,q)}$ does not depend on the choice of
$a_{i}$.

For $p \succ q$, we compactify $\mathcal{W}(p,q)$ to be $\overline{\mathcal{W}(p,q)}$
as follows.

Suppose $I_{1} = (p, r_{1}, \cdots, r_{s})$ and $I_{2} = (r_{s+1},
\cdots, r_{k},$ $q)$ are critical sequences such that $r_{s} \succeq
r_{s+1}$. Let $(I, s) = (p, r_{1}, \cdots, q)$. Note that $(I,s)$ is not necessarily a critical sequence because $r_{s}$ may equal $r_{s+1}$. Define
\[
\mathcal{W}_{I, s} = \mathcal{M}_{I_{1}} \times \mathcal{W}(r_{s},r_{s+1}) \times \mathcal{M}_{I_{2}}
\]
Here, if $r_{s} = r_{s+1}$, then by definition $\mathcal{W}(r_{s},r_{s+1}) = \{ r_{s} \}$. Furthermore, if $I_{1} = \{ p \}$, then $(I,s) = (I,0)$ and
\[
\mathcal{W}_{I,s} = \{ \beta(p) \} \times \mathcal{W}(r_{s},r_{s+1}) \times
\mathcal{M}_{I_{2}},
\]
we naturally identify $\mathcal{W}_{I,s}$ with $\mathcal{W}(r_{s},r_{s+1}) \times \mathcal{M}_{I_{2}}$. Similarly, if $I_{2} = \{ q \}$, we identify $\mathcal{W}_{I,s}$ with $\mathcal{M}_{I_{1}} \times \mathcal{W}(r_{s},r_{s+1})$. Most particularly, if $I_{1} = \{ p \}$ and $I_{2} = \{ q \}$, we identify $\mathcal{W}_{I,s} = \mathcal{W}_{I,0}$ with $\mathcal{W}(r_{s},r_{s+1}) = \mathcal{W} (p,q)$.

Define a space $\overline{\mathcal{W}(p,q)}$ as
\begin{equation}\label{equation_w(p,q)}
\overline{\mathcal{W}(p,q)} = \bigsqcup_{(I,s)} \mathcal{W}_{I, s},
\end{equation}
where the disjoint union is over all $(I,s) = (p, r_{1}, \cdots,
r_{k}, q)$ such that $p \succ r_{1} \succ \cdots \succ r_{s} \succeq
r_{s+1} \succ \cdots \succ r_{k} \succ q$.

Suppose $(\alpha_{1}, x, \alpha_{2}) \in \mathcal{M}_{I_{1}} \times
\mathcal{W}(r_{s},r_{s+1}) \times \mathcal{M}_{I_{2}} =
\mathcal{W}_{I, s}$. Then $x$ is on the unique generalized flow line
$\Gamma \in \overline{\mathcal{M}(p,q)}$ whose components include those of $\alpha_{1}$ and $\alpha_{2}$ and, in addition, the flow line through $x$. Thus, identify
$(\alpha_{1}, x, \alpha_{2})$ with $(\Gamma, x)$, we get
\[
\overline{\mathcal{W}(p,q)}  =  \{ (\Gamma, x) \in
\overline{\mathcal{M}(p,q)} \times M \mid \text{$x$ is on $\Gamma$} \}.
\]

\begin{definition}\label{definition_w(p,q)}
For $p \succ q$, define the set $\overline{\mathcal{W}(p,q)}$ as
(\ref{equation_w(p,q)}). Define the topology of
$\overline{\mathcal{W}(p,q)}$ as the restriction of that of
$\overline{\mathcal{M}(p,q)} \times M$. We call
$\overline{\mathcal{W}(p,q)}$ the compactified space of
$\mathcal{W}(p,q)$.
\end{definition}

Clearly, the map $\widetilde{E}: \overline{\mathcal{W}(p,q)}
\rightarrow \prod_{i=0}^{l} f^{-1}(a_{i}) \times M$ is a topological
embedding, where
\begin{equation}\label{embed_w(p,q)}
  \widetilde{E}(\Gamma, x) = (E(\Gamma), x).
\end{equation}
Thus the topology of $\overline{\mathcal{W}(p,q)}$ in this paper is
equivalent to that of \cite[thm. 3.5]{qin}.

Finally, we define the compactified space
$\overline{\mathcal{D}(p)}$ of $\mathcal{D}(p)$. Suppose $f$ is
bounded below.

Suppose $I = \{ p, r_{1}, \cdots, r_{k} \}$ is a critical sequence.
Define
\[
\mathcal{D}_{I} = \mathcal{M}_{I} \times \mathcal{D}(r_{k}).
\]
In particular, if $I = \{ p \}$, we naturally identify $\mathcal{D}_{I}$ with $\mathcal{D} (p)$.

Define a space $\overline{\mathcal{D}(p)}$ as
\begin{equation}\label{equation_d(p)}
\overline{\mathcal{D}(p)} = \bigsqcup_{I} \mathcal{D}_{I},
\end{equation}
where the disjoint union is over all critical sequences with head
$p$.

Suppose $(\alpha, x) \in \mathcal{M}_{I} \times \mathcal{D}(r_{k}) =
\mathcal{D}_{I}$. We can identify $\alpha$ with a generalized flow
line connecting $p$ with $r_{k}$. Adding the flow line passing
through $x$ to the above generalized flow line, we get a generalized
flow line connecting $p$ with $x$. Thus we get
\[
\overline{\mathcal{D}(p)}  =  \{ (\Gamma, x) \mid \Gamma \text{ is a
generalized flow line connecting $p$ with $x$} \}.
\]

The definition of the topology of $\overline{\mathcal{D}(p)}$ is
slightly complicated.

Suppose the critical values in $(-\infty, f(p)]$ are exactly $c_{l}
< \cdots < c_{0} = f(p)$. Define $U(i) \subseteq
\overline{\mathcal{D}(p)}$ ($i = 0, \cdots, l$) as
\begin{equation}\label{equation_u(i)}
  U(i) = \{ (\Gamma, x) \mid c_{i+1} < f(x) < c_{i-1} \},
\end{equation}
where $c_{l+1} = -\infty$ and $c_{-1} = +\infty$. Clearly,
$\overline{\mathcal{D}(p)} = \bigcup_{i} U(i)$. We have the
following injection $E_{i}: U(i) \rightarrow \prod_{j=0}^{i-1}
f^{-1}(a_{j}) \times M$ such that
\[
  E_{i}(\Gamma, x) = (x_{0}(\Gamma), \cdots, x_{i-1}(\Gamma), x),
\]
where $x_{j}(\Gamma)$ is the unique intersection point between
$\Gamma$ and $f^{-1}(a_{j})$. Equip $U(i)$ with the unique topology such
that $E_{i}$ is a topological embedding. The paper \cite[thm.
3.4]{qin} shows that these $U(i)$ have compatible smooth structures
under the assumption of the local triviality of the vector field.
Follow that argument, we can prove that the topologies of these
$U(i)$ are compatible even if we drop the local triviality. Here compatibility
means that $U(i)$ and $U(j)$ share the same topology on $U(i) \cap
U(j)$.

\begin{definition}\label{definition_d(p)}
Define the set $\overline{\mathcal{D}(p)}$ as (\ref{equation_d(p)}).
Define the topology of $\overline{\mathcal{D}(p)} = \bigcup_{i}
U(i)$ as the coherent topology such that each $U(i)$ is an open
subspace of $\overline{\mathcal{D}(p)}$ (see (\ref{equation_u(i)})).
We call $\overline{\mathcal{D}(p)}$ the compactified space of
$\mathcal{D}(p)$.
\end{definition}

For the convenience of the reader, we include here an example in \cite{qin}.

\begin{example}\label{example_d(p)_compactify}
Figure \ref{figure_d(p)_compactify} shows a standard example on a torus $T^{2} =
S^{1} \times S^{1}$, where the arrows indicate the directions of flows. Consider $S^{1}$ as the unit circle on the
complex plane. Define a Morse function on $T^{2}$ by $f(z_{1},
z_{2}) = \textrm{Re}(z_{1}) + \textrm{Re}(z_{2})$. $f$ has $4$
critical points $p$, $r$, $s$ and $q$. Their indices are $2$, $1$,
$1$ and $0$ respectively. Equip $T^{2}$ with the standard metric.
The left part of Figure \ref{figure_d(p)_compactify} shows the flow on $T^{2}$,
where the opposite sides of the square are identified with each
other. The right part is $\overline{\mathcal{D}(p)}$
which is an octagon. Here $\mathcal{M}(p,r) \times
\mathcal{D}(r)$ (or $\mathcal{M}(p,s) \times \mathcal{D}(s)$)
consists of open edges containing $r_{i}$ (or $s_{i}$), where
$i=1,2$; $\mathcal{M}(p,q) \times \mathcal{D}(q)$ consists of the
other $4$ open edges; and $(\mathcal{M}(p,r) \times \mathcal{M}(r,q)
\times \mathcal{D}(q)) \cup (\mathcal{M}(p,s) \times
\mathcal{M}(s,q) \times \mathcal{D}(q))$ consists of the $8$
vertices.
\end{example}

\begin{figure}[!htbp]
\centering
\includegraphics[scale=0.3]{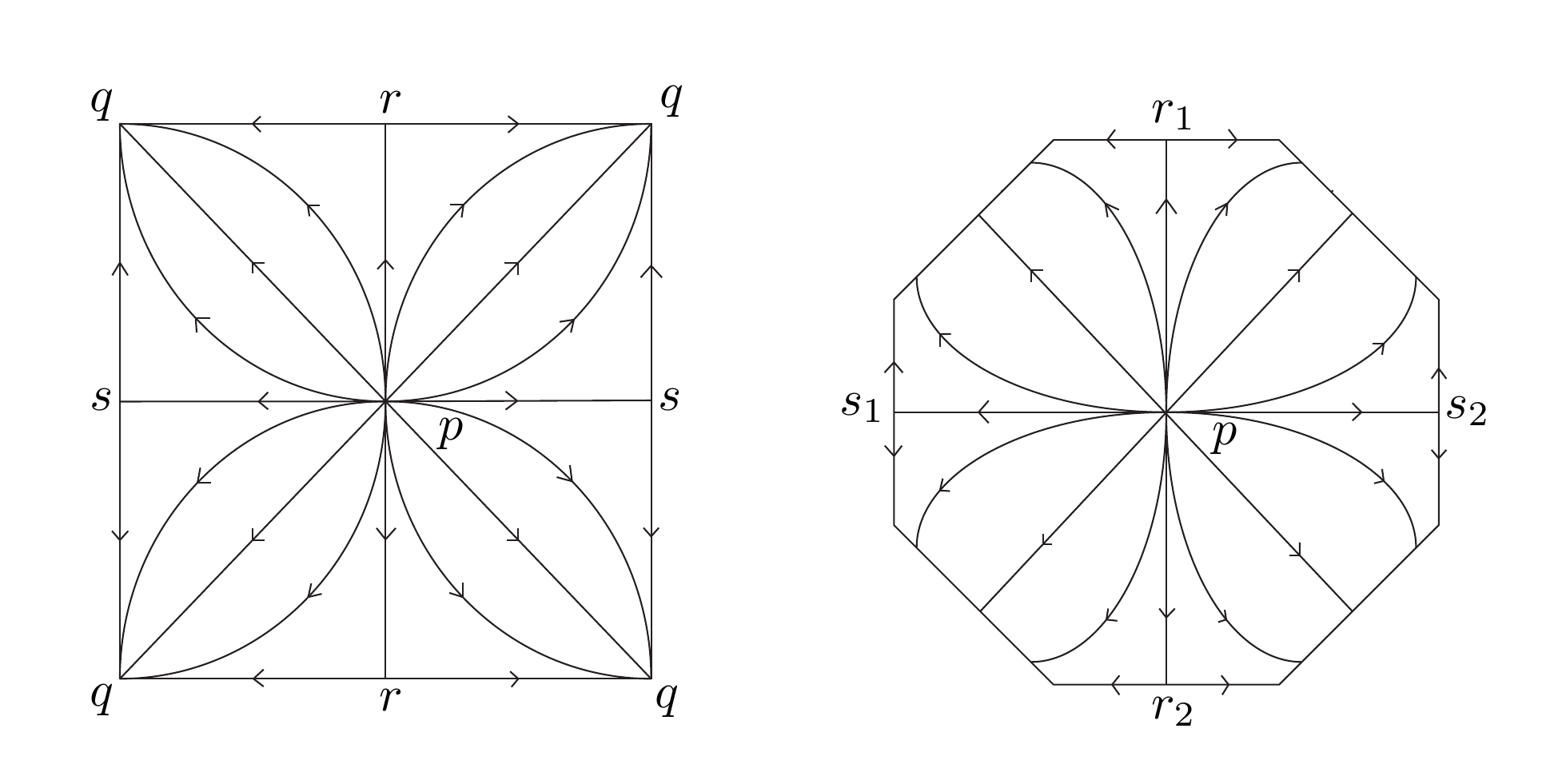} \caption{Compactification of the Descending Manifolds}
\label{figure_d(p)_compactify}
\end{figure}

Suppose $f_{1}$ and $f_{2}$ are Morse functions on $M_{1}$ and
$M_{2}$. Suppose $X_{i}$ is a negative gradient-like field for
$f_{i}$, and $X_{i}$ satisfies transversality. Suppose $h: M_{1}
\rightarrow M_{2}$ is a topological equivalence between $X_{1}$ and
$X_{2}$. If $p$ is a critical point of $f_{1}$, then $h(p)$ is a
critical point of $f_{2}$. Furthermore, $h(\mathcal{D}(p)) =
\mathcal{D}(h(p))$, $h(\mathcal{A}(p)) = \mathcal{A}(h(p))$, and
$h(\mathcal{W}(p,q)) = \mathcal{W}(h(p), h(q))$. Thus $h$ naturally
induces maps $h_{*}: \overline{\mathcal{M}(p,q)} \rightarrow
\overline{\mathcal{M}(h(p), h(q))}$, $h_{*}:
\overline{\mathcal{W}(p,q)} \rightarrow \overline{\mathcal{W}(h(p),
h(q))}$, and $h_{*}: \overline{\mathcal{D}(p)} \rightarrow
\overline{\mathcal{D}(h(p))}$. Here, if $\Gamma \in
\overline{\mathcal{M}(h(p), h(q))}$, then $h_{*}(\Gamma) =
h(\Gamma)$; if $(\Gamma, x) \in \overline{\mathcal{W}(p,q)}$ (or
$\overline{\mathcal{D}(p)}$), then $h_{*}(\Gamma, x) = (h(\Gamma),
h(x))$. Clearly, $h_{*}$ is a bijection and $(h_{*})^{-1} =
(h^{-1})_{*}$.

\begin{theorem}\label{theorem_topological_equivalence}
The maps $h_{*}: \overline{\mathcal{M}(p,q)} \rightarrow
\overline{\mathcal{M}(h(p), h(q))}$, $h_{*}:
\overline{\mathcal{D}(p)} \rightarrow \overline{\mathcal{D}(h(p))}$,
and $h_{*}: \overline{\mathcal{W}(p,q)} \rightarrow
\overline{\mathcal{W}(h(p), h(q))}$ are homeomorphisms.
\end{theorem}
\begin{proof}
It suffices to prove that $h_{*}$ is continuous because this implies
$h_{*}^{-1}$ is also continuous.

(1). We consider the case of $h_{*}: \overline{\mathcal{M}(p,q)}
\rightarrow \overline{\mathcal{M}(h(p), h(q))}$.

By the definition, $\overline{\mathcal{M}(p,q)}$ is identified with
a topological subspace of $\prod_{i=0}^{l} f_{1}^{-1}(a_{i})$ and
$\overline{\mathcal{M}(h(p), h(q))}$ is identified with a
topological subspace of $\prod_{i=0}^{k} f_{2}^{-1}(b_{i})$. By this
identification, for any $\Gamma \in \overline{\mathcal{M}(p,q)}$, we
have $\Gamma = (x_{0}(\Gamma), \cdots, x_{l}(\Gamma))$ and
$h_{*}(\Gamma) = (y_{0}(h(\Gamma)), \cdots, y_{k}(h(\Gamma)))$.
Suppose $x_{0}(\Gamma_{0})$ is on $\gamma \in \mathcal{M}(p,r)$ and
$\gamma$ is a component of $\Gamma_{0}$, then $h(x_{0}(\Gamma_{0}))$
is on $h(\gamma) \in \mathcal{M}(h(p), h(r))$. Suppose the regular
values in $[f_{2}(h(r)), f_{2}(h(p))]$ are $b_{0}, \cdots, b_{s}$.
Then $h(\gamma)$ intersects with $f_{2}^{-1}(b_{i})$ ($0 \leq i \leq
s$) at $y_{i}(h(\Gamma_{0}))$. When $\Gamma \in \overline{\mathcal{M}(p,q)}$ converges to
$\Gamma_{0}$, we have $x_{0}(\Gamma)$ converges to
$x_{0}(\Gamma_{0})$, then $h(x_{0}(\Gamma))$ converges to
$h(x_{0}(\Gamma_{0}))$. Thus, by Lemma \ref{lemma_flow_intersection}, when $\Gamma$ is close to $\Gamma_{0}$
enough, the flow line passing through $h(x_{0}(\Gamma))$  intersects
with $f_{2}^{-1}(b_{i})$ ($0 \leq i \leq s$) at $y_{i}(h(\Gamma))$
and $y_{i}(h(\Gamma))$ is continuous with respect to $\Gamma$.

By an induction, we can prove that, for all $0 \leq i \leq k$,
$y_{i}(h(\Gamma))$ is continuous with respect to $\Gamma$. Thus
$h_{*}$ is continuous.

(2). Since $\overline{\mathcal{W}(p,q)}$ is a topological subspace
of $\overline{\mathcal{M}(p,q)} \times M_{1}$, by (1), we infer that
$h_{*}$ is continuous on $\overline{\mathcal{W}(p,q)}$.

(3). We consider the case of $h_{*}: \overline{\mathcal{D}(p)}
\rightarrow \overline{\mathcal{D}(h(p))}$.

It suffices to check the continuity of $h_{*}$ on each $U(i)$.
Suppose $(\Gamma_{0}, z_{0}) \in U(i)$ and $\widetilde{c}_{s+1} <
f_{2}(h(z_{0})) < \widetilde{c}_{s-1}$, where $\widetilde{c}_{j}$
are critical values of $f_{2}$. Then $h_{*}(\Gamma_{0}, z_{0}) \in
\widetilde{U}(s)$, where $\widetilde{U}(s) \subseteq
\overline{\mathcal{D}(h(p))}$ is defined similarly to $U(i)$. Thus,
when $(\Gamma, z)$ is close to $(\Gamma_{0}, z_{0})$ enough, we have
$h_{*}(\Gamma, z) \in \widetilde{U}(s)$. Identify $\widetilde{U}(s)$
with a topological subspace of $\prod_{j=0}^{s-1} f_{2}^{-1}(b_{j})
\times M_{2}$, we have $h_{*}(\Gamma, z) = (y_{0}(h(\Gamma)),
\cdots, y_{s-1}(h(\Gamma)), h(z))$. By an argument similar to that
in (1), we can prove that $y_{j}(h(\Gamma))$ is continuous with
respect to $\Gamma$. Since $h(z)$ is continuous with respect to $z$,
we infer $h_{*}$ is continuous.
\end{proof}

\section{Properties of Moduli
Spaces}\label{section_property_moduli_spaces}
In this section, we establish the relevant properties of the
compactified moduli spaces. Particularly, the manifold structures of
these spaces will be emphasized.

When the metric is locally trivial, similar results can be found in
the literature (see e.g. \cite{latour}, \cite{burghelea_haller} and
\cite{qin}). Our results are extensions of those results to the case
of a general metric provided that the Morse function $f$ is proper.
In this case, every negative gradient-like vector field $X$ for $f$
satisfies the \textsl{CF} condition in \cite[def.\ 2.6]{qin}. This
extension needs Theorem \ref{theorem_regular_path}, Lemma
\ref{lemma_reduction}, and Theorem
\ref{theorem_topological_equivalence}.

We introduce the concepts of manifolds with corners or faces. Our
terminology follows that in \cite[p.\ 2]{douady}, \cite[sec.\
1.1]{janich} and \cite{qin}.

\begin{definition}\label{manifold_with_corner}
A smooth manifold with corners is a space defined in the same way as
a smooth manifold except that its atlases are open subsets of $[0, +
\infty)^{n}$.
\end{definition}

If $L$ is a smooth manifold with corners, $x \in L$, a neighborhood
of $x$ is diffeomorphic to $(0, \epsilon)^{n-k} \times [0,
\epsilon)^{k}$, then define $c(x) = k$. Clearly, $c(x)$ does not
depend on the choice of atlas.

\begin{definition}\label{definition_k_stratum}
Suppose $L$ is a smooth manifold. We call $\{ x \in L \mid c(x) = k
\}$ the $k$-stratum of $L$. Denote it by $\partial^{k} L$.
\end{definition}

Clearly, $\partial^{k} L$ is a submanifold \textit{without} corners
inside $L$, its codimension is $k$.

\begin{definition}\label{definition_manifold_with_face}
A smooth manifold $L$ with faces is a smooth manifold with corners
such that each $x$ belongs to the closures of $c(x)$ different
components of $\partial^{1} L$.
\end{definition}

Consider firstly the special case when $M$ is compact. By Theorem
\ref{theorem_regular_path}, we can construct a negative
gradient-like field $Y$ for $f$ such that $Y$ is locally trivial and
satisfies transversality. In addition, there exists a topological
equivalence between $X$ and $Y$ such that $h(p)=p$ for each critical
point $p$. Thus, by Theorem \ref{theorem_topological_equivalence},
$X$ and $Y$ have isomorphic compactified moduli spaces. Since the
properties of these spaces for $Y$ are proved in \cite{qin}, we
deduce certain properties of these spaces for $X$.

More generally, suppose that $f$ is proper but $M$ is not
necessarily compact. For any pair of critical points $(p,q)$, choose
regular values $a$ and $b$ such that $M^{a,b}$ is compact and
contains $p$ and $q$. By Lemma \ref{lemma_reduction}, we can embed
$M^{a,b}$ into $\widetilde{M}$, extend $f|_{M^{a,b}}$ to be
$\widetilde{f}$ on $\widetilde{M}$, and extend $X|_{M^{a,b}}$ to be
$\widetilde{X}$ on $\widetilde{M}$. Furthermore, $\mathcal{W}(p,q;X)
= \mathcal{W}(p,q;\widetilde{X})$. Thus we get
$\overline{\mathcal{M}(p,q;X)} =
\overline{\mathcal{M}(p,q;\widetilde{X})}$ and
$\overline{\mathcal{W}(p,q;X)} =
\overline{\mathcal{W}(p,q;\widetilde{X})}$. If $f$ is bounded below,
we choose $M^{a}$ such that $p \in M^{a}$. Do the above extension
again to get $\overline{\mathcal{D}(p;X)} =
\overline{\mathcal{D}(p;\widetilde{X})}$. Thus Lemma
\ref{lemma_reduction} reduces the proper case to the compact case.

Before formulating the property of $\overline{\mathcal{M}(p,q)}$, we
introduce a map. Suppose $\Gamma_{1} \in
\overline{\mathcal{M}(p,r)}$ is a generalized flow line connecting
$p$ with $r$ and $\Gamma_{2} \in \overline{\mathcal{M}(r,q)}$ is a
generalized flow line connecting $r$ with $q$. Thus the combination
of $\Gamma_{1}$ and $\Gamma_{2}$ gives a generalized flow line
$\Gamma$ connecting $p$ with $q$. So we have the natural inclusion
$i_{(p,r,q)}: \overline{\mathcal{M}(p,r)} \times
\overline{\mathcal{M}(r,q)} \rightarrow
\overline{\mathcal{M}(p,q)}$.

\begin{theorem}\label{theorem_m(p,q)}
Suppose $f$ is proper and $X$ satisfies transversality. Then, for
each critical sequence $\{ p, q \}$, the topological space
$\overline{\mathcal{M}(p,q)} = \bigsqcup \mathcal{M}_{I}$ is defined
by Definition \ref{definition_m(p,q)}. It has the following
properties.

(1). It is a compact topological manifold with boundary. Its
interior is $\mathcal{M}(p,q)$.

(2). Its topology is compatible with that of each $\mathcal{M}_{I}$, and
the map $i_{(p,r,q)}: \overline{\mathcal{M}(p,r)} \times
\overline{\mathcal{M}(r,q)} \rightarrow \overline{\mathcal{M}(p,q)}$
is a topological embedding.

(3). The evaluation map $E: \overline{\mathcal{M}(p,q)} \rightarrow
\prod_{i=0}^{l} f^{-1}(a_{i})$ is a topological embedding, where $E$
is defined in (\ref{evalution_m(p,q)}).

(4). There exists a topological embedding $\iota:
\overline{\mathcal{M}(p,q)} \rightarrow \prod_{i=0}^{l}
f^{-1}(a_{i})$ such that $\iota(\overline{\mathcal{M}(p,q)})$ is a
smoothly embedded submanifold with faces inside $\prod_{i=0}^{l}
f^{-1}(a_{i})$ and the $k$-stratum of
$\iota(\overline{\mathcal{M}(p,q)})$ is $\bigsqcup_{|I|=k}
\iota(\mathcal{M}_{I})$.

In particular, if $M$ is compact, then there exist homeomorphisms
$h_{i}: f^{-1}(a_{i}) \rightarrow f^{-1}(a_{i})$ such that $\iota =
(\prod_{i=0}^{l} h_{i}) \circ E$ in (4).
\end{theorem}

\begin{theorem}\label{theorem_compact_m(p,q)}
Under the assumption of Theorem \ref{theorem_m(p,q)}, each
$\overline{\mathcal{M}(p,q)}$ carries a smooth structure compatible
with its topology such that $\overline{\mathcal{M}(p,q)}$ is a
compact smooth manifold with faces and $\partial^{k}
\overline{\mathcal{M}(p,q)} = \bigsqcup_{|I|=k} \mathcal{M}_{I}$. In
particular, suppose $M$ is compact, then $i_{(p,r,q)}:
\overline{\mathcal{M}(p,r)} \times \overline{\mathcal{M}(r,q)}
\rightarrow \overline{\mathcal{M}(p,q)}$ is a smooth embedding.
\end{theorem}

\begin{remark}
The (1) of Theorem \ref{theorem_m(p,q)} shows that we can add a
boundary to $\mathcal{M}(p,q)$ such that it becomes a compact
manifold with boundary. The following theorems show that this is
also true for $\mathcal{W}(p,q)$ and $\mathcal{D}(p)$. Thus moduli
spaces are \textit{special} open manifolds (if they are open)
because there exist obstructions of adding a boundary to a
general open manifold (see \cite{siebenmann}).
\end{remark}

\begin{remark}\label{remark_embed_m(p,q)}
The paper \cite[example 3.1]{qin} shows that, if the metric is not
locally trivial, then $E(\overline{\mathcal{M}(p,q)})$ usually is
even not a $C^{1}$ embedded submanifold of $\prod_{i=0}^{l}
f^{-1}(a_{i})$. Here $E$ is the evaluation map in the (3) of Theorem
\ref{theorem_m(p,q)}. However, the (4) of Theorem
\ref{theorem_m(p,q)} shows that a suitable embedding $\iota$ makes
the image good.
\end{remark}

\begin{proof}[Proof of Theorem \ref{theorem_m(p,q)}]
Choose regular values $a$ and $b$ such that $M^{a,b}$ is compact and
contains $p$ and $q$. As described in the above, construct
$\widetilde{M}$, $\widetilde{f}$ and $\widetilde{X}$. We have
$\overline{\mathcal{M}(p,q;X)} =
\overline{\mathcal{M}(p,q;\widetilde{X})}$ and
$\mathcal{M}_{I}(\widetilde{X}) = \mathcal{M}_{I}(X)$ for all
critical sequences $I$ with head $p$ and tail $q$. There exists a negative gradient-like vector field $Y$ for $\widetilde{f}$ on $\widetilde{M}$ and a topological equivalence $h: \widetilde{M} \rightarrow \widetilde{M}$
which maps the orbits of $\widetilde{X}$ to those of $Y$, where $Y$
is locally trivial.

(1). By \cite[thm.\ 3.3]{qin}, we know that
$\overline{\mathcal{M}(p,q;Y)}$ is a compact smooth manifold with
faces whose $k$-stratum is $\bigsqcup_{|I|=k} \mathcal{M}_{I}(Y)$.
Thus $\overline{\mathcal{M}(p,q;Y)}$ is a compact topological
manifold with boundary, and its interior is $\mathcal{M}(p,q;Y)$. By
Theorem \ref{theorem_topological_equivalence}, we know that $h$
induces a homeomorphism $h_{*}:
\overline{\mathcal{M}(p,q;\widetilde{X})} \rightarrow
\overline{\mathcal{M}(p,q;Y)}$ such that
$h_{*}(\mathcal{M}_{I}(\widetilde{X})) = \mathcal{M}_{I}(Y)$. This
completes the proof of (1).

(2). The proof is easy and even does not need the comparison among
$\overline{\mathcal{M}(p,q;X)}$,
$\overline{\mathcal{M}(p,q;\widetilde{X})}$ and
$\overline{\mathcal{M}(p,q;Y)}$. Similar details is also included in
the proof of \cite[thm.\ 3.3]{qin}.

(3). This is the definition of the topology of
$\overline{\mathcal{M}(p,q;X)}$.

(4). Let $E_{Y}: \overline{\mathcal{M}(p,q;Y)} \rightarrow
\prod_{i=0}^{l} \widetilde{f}^{-1}(a_{i})$ be the evaluation map. By
\cite[thm. 3.3]{qin}, we know $E_{Y}$ is a smooth embedding.

Suppose $r_{1}$ and $r_{2}$ are in $M^{a,b}$. It's easy to see that $\mathcal{D}(r_{1}; Y) \subseteq M^{a,b} \cup \widetilde{f}^{-1}((- \infty, a))$ and $\mathcal{A}(r_{2}; Y) \subseteq M^{a,b} \cup \widetilde{f}^{-1}((b, + \infty))$. We infer that $\mathcal{W}(r_{1}, r_{2}; Y) \subseteq M^{a,b}$. Thus $\text{Im}(E_{Y}) \subseteq \prod_{i=0}^{l} f^{-1}(a_{i})$ and
$\iota = E_{Y} \circ h_{*}$ is the desired map.

Finally, we consider the special case when $M$ is compact. We construct $Y$ on $M$. The topological equivalence $h: M
\rightarrow M$ induces the homeomorphism $h_{*}:
\overline{\mathcal{M}(p,q;X)} \rightarrow
\overline{\mathcal{M}(p,q;Y)}$. We consider the relation between
$h(f^{-1}(a_{i}))$ and $f^{-1}(a_{i})$. Denote by $\phi_{t}^{X}$ the
flow generated by $X$ and by $\phi_{t}^{Y}$ the flow generated by
$Y$. For any $x \in f^{-1}(a_{i})$, we have $\phi^{X}(-\infty, x) =
r_{1}$ for some $r_{1} \in M - M^{a_{i}}$ and $\phi^{X}(+\infty, x)
= r_{2}$ for some $r_{2} \in M^{a_{i}}$. Since $h$ is a topological
equivalence fixing $r_{1}$ and $r_{2}$, we know that
$\phi^{Y}(-\infty, h(x)) = r_{1}$ and $\phi^{Y}(+\infty, h(x)) =
r_{2}$. Thus, $\phi^{Y}(t(x), h(x)) \in f^{-1}(a_{i})$ for some $t(x) \in
(-\infty, +\infty)$. By Lemma \ref{lemma_flow_intersection}, an isotopy along the flows generated by $Y$
gives a homeomorphism $\psi_{i}: h(f^{-1}(a_{i})) \rightarrow
f^{-1}(a_{i})$. We complete the proof by defining $h_{i} = \psi_{i}
\circ h$.
\end{proof}

\begin{proof}[Proof of Theorem \ref{theorem_compact_m(p,q)}]
The first half part of Theorem \ref{theorem_compact_m(p,q)} is a
corollary of Theorem \ref{theorem_m(p,q)}. It remains to prove that, when $M$ is compact, there exists a suitable smooth structure of $\overline{\mathcal{M}(r_{1}, r_{2}; X)}$ for each pair of critical points $r_{1}$ and $r_{2}$ such that $i_{(p,r,q)}$ is a smooth embedding.

Construct a locally trivial field $Y$ as we did in the proof of Theorem \ref{theorem_m(p,q)}. We get a topological equivalence $h: M \rightarrow M$ which induces homeomorphisms $h_{*}^{r_{1}, r_{2}}: \overline{\mathcal{M}(r_{1}, r_{2}; X)} \rightarrow \overline{\mathcal{M}(r_{1}, r_{2}; Y)}$. Here we use the notation $h_{*}^{r_{1}, r_{2}}$ instead of $h_{*}$ to indicate the critical points. By \cite[thm.\ 3.3]{qin}, each $\overline{\mathcal{M}(r_{1}, r_{2}; Y)}$ has a natural smooth structure. Define the smooth structure of $\overline{\mathcal{M}(r_{1}, r_{2}; X)}$ such that $h_{*}^{r_{1}, r_{2}}$ is a diffeomorphism. We have the following commutative diagram:
\[
\xymatrix{
  \overline{\mathcal{M}(p,r;X)} \times \overline{\mathcal{M}(r,q;X)} \ar[d]_{h_{*}^{p,r} \times h_{*}^{r,q}} \ar[r]^-{i_{(p,r,q)}} & \overline{\mathcal{M}(p,q;X)} \ar[d]^{h_{*}^{p,q}} \\
  \overline{\mathcal{M}(p,r;Y)} \times \overline{\mathcal{M}(r,q;Y)} \ar[r]^-{i'_{(p,r,q)}} & \overline{\mathcal{M}(p,q;Y)} \, , }
\]
where $i'_{(p,r,q)}$ is the natural inclusion similar to $i_{(p,r,q)}$. By \cite[thm.\ 3.3]{qin}, we know that $i'_{(p,r,q)}$ is a smooth embedding. Since the vertical maps of this diagram are diffeomorphisms, we infer that $i_{(p,r,q)}$ is a smooth embedding.
\end{proof}

Since we define $\overline{\mathcal{W}(p,q)}$ as a subspace
of $\overline{\mathcal{M}(p,q)} \times M$, we have the inclusion $i:
\overline{\mathcal{W}(p,q)} \rightarrow \overline{\mathcal{M}(p,q)}
\times M$. Suppose $p \succ r \succ q$. If $(\Gamma_{1}, x) \in \overline{\mathcal{W}(p,r)}$ and
$\Gamma_{2} \in \overline{\mathcal{M}(r,q)}$, then the combination
of $\Gamma_{1}$ and $\Gamma_{2}$ gives an element in
$\overline{\mathcal{M}(p,q)}$ and $x$ is on it. This defines a
natural inclusion $i_{(p,r,q)}^{1}: \overline{\mathcal{W}(p,r)}
\times \overline{\mathcal{M}(r,q)} \rightarrow
\overline{\mathcal{W}(p,q)}$. Similarly, we can define a natural
inclusion $i_{(p,r,q)}^{2}: \overline{\mathcal{M}(p,r)} \times
\overline{\mathcal{W}(r,q)} \rightarrow
\overline{\mathcal{W}(p,q)}$.

Suppose $f$ is bounded below. We define the evaluation map $e:
\overline{\mathcal{D}(p)} \rightarrow M$ as
\begin{equation}
e(\Gamma, x) = x.
\end{equation}
Clearly, the restriction of $e$ to $\mathcal{D}_{I} =
\mathcal{M}_{I} \times \mathcal{D}(r_{k})$ is the coordinate
projection onto $ \mathcal{D}(r_{k}) \subseteq M$.

\begin{example}
In Example \ref{example_d(p)_compactify}, the space $\overline{\mathcal{D}(p)}$ is illustrated by the right part of Figure \ref{figure_d(p)_compactify}. The evaluation map $e$ maps the interior of $\overline{\mathcal{D}(p)}$ to $\mathcal{D}(p)$. On the boundary of $\overline{\mathcal{D}(p)}$, it maps the horizontal open line segment through $r_{i}$ to $\mathcal{D}(r)$, maps the vertical open line segment through $s_{i}$ to $\mathcal{D}(s)$, ($i=1,2$), and maps the remaining part to $q$. Furthermore, it maps $r_{i}$ to $r$ and maps $s_{i}$ to $s$.
\end{example}

Suppose $p \succ r$. If $\Gamma_{1} \in \overline{\mathcal{M}(p,r)}$ and $(\Gamma_{2}, x) \in
\overline{\mathcal{D}(r)}$, then the combination of $\Gamma_{1}$ and
$\Gamma_{2}$ is a generalized flow line connecting $p$ with $x$.
This defines a natural inclusion $i_{(p,r)}:
\overline{\mathcal{M}(p,r)} \times \overline{\mathcal{D}(r)}
\rightarrow \overline{\mathcal{D}(p)}$.

By \cite[thms.\ 3.4, 3.5 and 3.7]{qin}, using an argument similar to
the proof of Theorem \ref{theorem_m(p,q)}, we can get the following
results. The proof of Theorem \ref{theorem_w(p,q)} needs the fact
that the map $\widetilde{E}$ defined in (\ref{embed_w(p,q)}) is a
smooth embedding when the vector field $X$ is locally trivial.
Although this fact is not stated in \cite{qin}, its easy to see that
it is true from the proof of \cite[thm.\ 3.5]{qin}.

\begin{theorem}\label{theorem_w(p,q)}
Suppose $f$ is proper and $X$ satisfies transversality. Then, for
each critical sequence $\{ p,q \}$, the topological space
$\overline{\mathcal{W}(p,q)} = \bigsqcup_{(I,s)} \mathcal{W}_{I,s}$
is defined by Definition \ref{definition_w(p,q)}. It has the following
properties.

(1). It is a compact topological manifold with boundary. Its
interior is $\mathcal{W}(p,q)$.

(2). Its topology is compatible with that of each $\mathcal{W}_{I,s}$.
The maps $i_{(p,r,q)}^{1}: \overline{\mathcal{W}(p,r)} \times
\overline{\mathcal{M}(r,q)} \rightarrow \overline{\mathcal{W}(p,q)}$
and $i_{(p,r,q)}^{2}: \overline{\mathcal{M}(p,r)} \times
\overline{\mathcal{W}(r,q)} \rightarrow \overline{\mathcal{W}(p,q)}$
are topological embeddings.

(3). The inclusion $i: \overline{\mathcal{W}(p,q)} \rightarrow
\overline{\mathcal{M}(p,q)} \times M$ and the map $\widetilde{E}:
\overline{\mathcal{W}(p,q)} \rightarrow \prod_{i=0}^{l}
f^{-1}(a_{i}) \times M$ are topological embeddings, where
$\widetilde{E}$ is defined in (\ref{embed_w(p,q)}).

(4). There exists a topological embedding $\iota:
\overline{\mathcal{W}(p,q)} \rightarrow \prod_{i=0}^{l}
f^{-1}(a_{i}) \times M$ such that
$\iota(\overline{\mathcal{W}(p,q)})$ is a smoothly embedded
submanifold with faces inside $\prod_{i=0}^{l} f^{-1}(a_{i}) \times
M$ and the $k$-stratum of $\iota(\overline{\mathcal{W}(p,q)})$ is
$\bigsqcup_{(I,s)} \iota(\mathcal{W}_{I,s})$, where $(I,s)$ contains
$k+2$ components.

In particular, if $M$ is compact, then there exist homeomorphisms
$h_{i}: f^{-1}(a_{i}) \rightarrow f^{-1}(a_{i})$ such that $\iota =
[(\prod_{i=0}^{l} h_{i}) \times h] \circ \widetilde{E}$ in (4).
\end{theorem}

\begin{corollary}
Under the assumption of Theorem \ref{theorem_w(p,q)},
$\overline{\mathcal{W}(p,q)}$ carries a smooth structure compatible
with its topology such that $\overline{\mathcal{W}(p,q)}$ is a
compact smooth manifold with faces and $\partial^{k}
\overline{\mathcal{W}(p,q)} = \bigsqcup_{(I,s)} \mathcal{W}_{I,s}$,
where $(I,s)$ contains $k+2$ components.
\end{corollary}

\begin{theorem}\label{theorem_d(p)}
Suppose $f$ is proper and bounded below. Suppose $X$ satisfies
transversality. Then, for each critical point $p$, the topological space
$\overline{\mathcal{D}(p)} = \bigsqcup \mathcal{D}_{I}$ is defined
by Definition \ref{definition_d(p)}. It has the following
properties.

(1). It is homeomorphic to a closed disc. Its interior is
$\mathcal{D}(p)$.

(2). Its topology is compatible with that of each $\mathcal{D}_{I}$. The
map $i_{(p,r)}: \overline{\mathcal{M}(p,r)} \times
\overline{\mathcal{D}(r)} \rightarrow \overline{\mathcal{D}(p)}$ is
a topological embedding.

(3). The evaluation map $e: \overline{\mathcal{D}(p)} \rightarrow M$
is continuous. Here, for a critical sequence $I$ with tail $r_{k}$, the restriction of $e$ to $\mathcal{D}_{I} = \mathcal{M}_{I} \times \mathcal{D}(r_{k})$ is given by
\begin{eqnarray}\label{theorem_d(p)_1}
e|_{\mathcal{D}_{I}}: \mathcal{M}_{I} \times \mathcal{D}(r_{k}) & \rightarrow & \mathcal{D}(r_{k}) \subseteq M \\
e(\alpha, x) & = & x,  \nonumber
\end{eqnarray}
i.e. $e|_{\mathcal{D}_{I}}$ is the coordinate projection onto $ \mathcal{D}(r_{k})$. In particular, $e|_{\mathcal{D}(p)}$ is the identity inclusion.

(4). It carries a smooth structure compatible with its topology such
that it is a compact smooth manifold with faces and $\partial^{k}
\overline{\mathcal{D}(p)} = \bigsqcup_{|I|=k-1} \mathcal{D}_{I}$.
\end{theorem}

\section{Orientation Formulas}\label{section_orientation_formulas}
In this section, we shall prove the following orientation formulas.

\begin{theorem}[Orientation Formulas]\label{theorem_orientation}
Suppose $f$ is proper and $X$ satisfies transversality. As oriented
topological manifolds, we have

(1). $\displaystyle \partial^{1} \overline{\mathcal{M}(p,q)} =
\bigsqcup_{p \succ r \succ q} (-1)^{\mathrm{ind}(p) -
\mathrm{ind}(r)} \mathcal{M}(p,r) \times \mathcal{M}(r,q)$;

(2). $\displaystyle \partial^{1} \overline{\mathcal{D}(p)} =
\bigsqcup_{p \succ r} \mathcal{M}(p,r) \times \mathcal{D}(r)$, where
$f$ is bounded below;

(3). $\displaystyle \partial^{1} \overline{\mathcal{W}(p,q)} =
\bigsqcup_{p \succeq r \succ q} (-1)^{\mathrm{ind}(p) -
\mathrm{ind}(r) + 1} \mathcal{W}(p,r) \times \mathcal{M}(r,q) \sqcup
\bigsqcup_{p \succ r \succeq q} \mathcal{M}(p,r) \times
\mathcal{W}(r,q)$.

In the above, $\partial^{1} \Box$ are equipped with boundary
orientations, $\Box \times \Box$ are equipped with product
orientations, and $\mathrm{ind}(*)$ is the Morse index of $*$.
\end{theorem}

In order to explain the concepts in Theorem
\ref{theorem_orientation}, we need to review the definition of
orientation at first.

Suppose $M$ is an $n$ dimensional smooth manifold. In algebraic
topology, the orientation of $M$ at $x$ is a generator $\alpha \in
H^{n}(M, M-\{x\})$. In differential topology, the orientation is an
ordered base $\{ e_{1}, \cdots, e_{n} \} \subseteq T_{x}M$. These
two definitions are related as follows. Choose a smooth embedding
$\varphi: V \rightarrow M$ such that $\varphi(0) =x$ and $D
\varphi(0) = \text{Id}$, where $V$ is a neighborhood of $0$ in
$T_{x}M$. Then $\varphi^{*} \alpha \in H^{n}(V,V-\{0\}) =
H^{n}(T_{x}M, T_{x}M -\{0\})$ is a generator. Here $\varphi^{*}
\alpha$ does not depend on the choice of $\varphi$. Actually, if
$\widetilde{\varphi}$ is another such embedding, then there exists
an isotopy between $\varphi$ and $\widetilde{\varphi}$ in a smaller
neighborhood of $0$. Denote by $\alpha_{0}$ the preferred generator
in $H^{n}(\mathbb{R}^{n}, \mathbb{R}^{n}-\{0\})$ (see \cite[p.\
266]{milnor_stasheff}). The ordered base $\{ e_{1}, \cdots, e_{n}
\}$ determines a linear isomorphism $A: T_{x}M \rightarrow
\mathbb{R}^{n}$, then $A^{*} \alpha_{0} \in H^{n}(T_{x}M, T_{x}M
-\{0\})$ is also a generator. We say that these two definitions give
the same orientation if and only if $\varphi^{*} \alpha = A^{*}
\alpha_{0}$.

Suppose $L$ is a $k$ dimensional embedded submanifold of $M$ such
that its normal bundle is orientable. Choose a neighborhood $U$ of
$L$ such that $L$ is closed in $U$. Choose a Thom class $\beta \in
H^{n-k}(U, U-L)$. The Thom class $\beta$ defines the normal
orientation in the sense of algebraic topology. On the other hand,
for any $x \in L$, choose an ordered base $\{ \varepsilon_{k+1},
\cdots, \varepsilon_{n} \}$ of the normal space $N_{x}(L,M) =
T_{x}M/T_{x}L$. This defines the normal orientation of $L$ at $x$ in
the sense of differential topology. These two definitions are
related as follows. Let $\varphi: V \rightarrow M$ be a smooth
embedding such that $\varphi(0)=x$ and $P \cdot D \varphi(0) =
\text{Id}$, where $V$ is a neighborhood of $0$ in $N_{x}(L,M)$ and
$P: T_{x}M \rightarrow T_{x}M/T_{x}L = N_{x}(L,M)$ is the
projection. Then $\varphi^{*} \beta \in H^{n-k}(V,V-\{0\}) =
H^{n-k}(N_{x},N_{x}-\{0\})$ is a generator. Here $\varphi^{*} \beta$
does not depend on the choice of $\varphi$. The ordered base
determines an isomorphism $A: N_{x} \rightarrow \mathbb{R}^{n-k}$.
So $A^{*} \alpha_{0}$ is also a generator of
$H^{n-k}(N_{x},N_{x}-\{0\})$, where $\alpha_{0}$ is the preferred
generator of $H^{n-k}(\mathbb{R}^{n-k}, \mathbb{R}^{n-k}-\{0\})$.
These two definitions coincide if and only if $\varphi^{*} \beta =
A^{*} \alpha_{0}$.

Suppose $\{ e_{1}, \cdots, e_{k} \} \subseteq T_{x}L$ represents the
orientation of $L$ and $\{ e_{k+1}, \cdots, e_{n} \} \subseteq
T_{x}M$ represents the normal orientation of $L$. We say the
orientation $\{ e_{1}, \cdots, e_{n} \}$ of $M$ is defined by the
orientation and the normal orientation of $L$. We have the following
lemma whose proof is in the Appendix.

\begin{lemma}\label{lemma_preserve_orientation}
Suppose $M_{i}$ ($i=1,2$) is a smooth orientable manifold, $L_{i}$
is an orientable submanifold and a closed subset of $M_{i}$. Suppose the orientation and the normal
orientation of $L_{i}$ define the orientation of $M_{i}$. Let
$\beta_{i} \in H^{n-k}(M_{i}, M_{i}-L_{i})$ be the Thom class
representing the normal orientation of $L_{i}$. Let $h: (M_{1},
L_{1}) \rightarrow (M_{2}, L_{2})$ be a homeomorphism such that $h$
preserves the orientation of $M_{i}$ and $h^{*} \beta_{2} =
\beta_{1}$. Then $h$ preserves the orientation of $L_{i}$.
\end{lemma}

Following \cite{qin}, we define the orientations of
$\mathcal{D}(p)$, $\mathcal{W}(p,q)$ and $\mathcal{M}(p,q)$. We
review the definition by means of differential topology in \cite[p.\
500]{qin} as follows (see \cite{qin} for more details).

Assign an arbitrary orientation to $\mathcal{D}(p)$ for each
critical point $p$. Since $\mathcal{D}(q)$ and $\mathcal{A}(q)$ are
transversal, the orientation of $\mathcal{D}(q)$ gives the normal
bundle $N(\mathcal{A}(q), M) = T_{\mathcal{A}(q)} M / T
\mathcal{A}(q)$ an orientation. Since $\mathcal{D}(p)$ is transverse
to $\mathcal{A}(q)$ and $\mathcal{W}(p,q) = \mathcal{D}(p) \cap
\mathcal{A}(q)$, the orientation of $N(\mathcal{A}(q), M)$ gives the
normal bundle $N(\mathcal{W}(p,q), \mathcal{D}(p))$ an orientation.
We choose the orientation of $\mathcal{W}(p,q)$ such that the
orientation and the normal orientation of $\mathcal{W}(p,q)$ define
the orientation of $\mathcal{D}(p)$. Identify $\mathcal{M}(p,q)$
with $\mathcal{W}(p,q) \cap f^{-1}(a)$ for some regular value $a \in
(f(q), f(p))$. The orientation of $\mathcal{W}(p,q) \cap f^{-1}(a)$
is defined by the direction of the flow and the orientation of
$\mathcal{W}(p,q)$. This defines the orientation of
$\mathcal{M}(p,q)$. This definition does not depend on the choice of
$a$.

By Theorems \ref{theorem_m(p,q)}, we know
$\overline{\mathcal{M}(p,q)}$ is a topological manifold with
boundary, whose interior is $\mathcal{M}(p,q)$. Thus the orientation
of $\mathcal{M}(p,q)$ gives $\partial \overline{\mathcal{M}(p,q)}$
the \textbf{boundary orientation} in the usual sense. In differential topology, the combination of the outward normal direction and the
boundary orientation of the boundary gives the orientation of the
manifold. In algebraic topology, the boundary orientation is defined by \cite[(28.7), (28.16)]{greenberg_harper}. Also by Theorem \ref{theorem_m(p,q)}, we know that
$\partial^{1} \overline{\mathcal{M}(p,q)} = \underset{|I|=1}{\sqcup}
\mathcal{M}_{I} = \underset{p \succ r \succ q}{\sqcup}
\mathcal{M}(p,r) \times \mathcal{M}(r,q)$ is an open subset of
$\partial \overline{\mathcal{M}(p,q)}$. Thus $\partial^{1}
\overline{\mathcal{M}(p,q)}$ has the boundary orientation. On the
other hand, both $\mathcal{M}(p,r)$ and $\mathcal{M}(r,q)$ have
orientations. Thus $\mathcal{M}(p,r) \times \mathcal{M}(r,q)$ has
the \textbf{product orientation}. We shall consider the relation
between these two orientations. Similarly,
$\overline{\mathcal{D}(p)}$ and $\overline{\mathcal{M}(p,q)}$ also
have such issues. Theorem \ref{theorem_orientation} indicates these
relations.

Similarly to the previous section, by Lemma \ref{lemma_reduction},
we may assume that $M$ is compact. By Theorem
\ref{theorem_regular_path}, we can construct the locally trivial
field $Y$ and the topological equivalence $h$ mapping the orbits of
$X$ to those of $Y$.

However, since $h$ is not assumed differentiable, we have to use the algebraic
method to describe the orientation of $\mathcal{W}(p,q;X)$ again.
Choose an open tubular neighborhood $U_{q}$ of $\mathcal{A}(q;X)$
such that $\mathcal{A}(q;X)$ is closed in $U_{q}$. Suppose the index
$\mathrm{ind}(q) = s$. We have the inclusion isomorphism
\[
\xymatrix@C=0.2cm{
  H^{s}(U_{q}, U_{q} - \mathcal{A}(q;X)) \ar[rr]^-{\cong}
  && H^{s}(U_{q} \cap \mathcal{D}(q;X), U_{q} \cap \mathcal{D}(q;X) - \{q\}),
}
\]
where $H^{s}(U_{q} \cap \mathcal{D}(q;X), U_{q} \cap
\mathcal{D}(q;X) - \{q\}) = H^{s}(\mathcal{D}(q;X), \mathcal{D}(q;X)
- \{q\})$. Thus the orientation of $\mathcal{D}(q;X)$, $\alpha_{q}
\in H^{s}(\mathcal{D}(q;X), \mathcal{D}(q;X) - \{q\})$, determines a
Thom class $\beta_{q} \in H^{s}(U_{q}, U_{q} - \mathcal{A}(q;X))$.
Let $U_{p,q} = \mathcal{D}(p;X) \cap U_{q}$. Then $U_{p,q}$ is open
in $\mathcal{D}(p;X)$ and $\mathcal{W}(p,q;X)$ is closed in
$U_{p,q}$. By the inclusion monomorphism (it is an isomorphism if
and only if $U_{p,q}$ is connected)
\[
\xymatrix@C=0.2cm{
  H^{s}(U_{q}, U_{q} - \mathcal{A}(q;X)) \ar[rr] && H^{s}(U_{p,q}, U_{p,q} -\mathcal{W}(p,q;X)), }
\]
we have that $\beta_{q}$ determines a Thom class $\beta_{p,q} \in
H^{s}(U_{p,q}, U_{p,q} -\mathcal{W}(p,q;X))$. Clearly, $U_{p,q}$
inherits the orientation from $\mathcal{D}(p;X)$. Thus $\beta_{p,q}$
and the orientation of $U_{p,q}$ give $\mathcal{W}(p,q;X)$ the
orientation.

Since $h(\mathcal{D}(p;X)) = \mathcal{D}(p;Y)$, we can define the
orientation of $\mathcal{D}(p;Y)$ as $\alpha_{p}' = (h^{-1})^{*}
\alpha_{p} \in H^{s}(\mathcal{D}(p;Y), \mathcal{D}(p;Y) - \{p\})$
for each $p$. Then the orientations of $\mathcal{W}(p,q;$ $Y)$ are
defined. We also have $h(\mathcal{W}(p,q;X)) = \mathcal{W}(p,q;Y)$.

\begin{lemma}\label{lemma_w_orientation}
The topological equivalence $h$ preserves the orientation of
$\mathcal{W}(p,q;X)$.
\end{lemma}
\begin{proof}
Choose the open tubular neighborhood $U_{q}$ of $\mathcal{A}(q;X)$
and define $U_{p,q} = \mathcal{D}(p;X) \cap U_{q}$ as the above.
Define $U_{q}' = h(U_{q})$ and $U_{p,q}' = h(U_{p,q})$. We may
assume $U_{p,q}$ is connected.

Suppose the orientation of $\mathcal{D}(q;Y)$ defines the Thom class
$\beta_{q}' \in H^{s}(U_{q}', U_{q}' - \mathcal{A}(q;Y))$ and the
Thom class $\beta_{p,q}' \in H^{s}(U_{p,q}', U_{p,q}' -
\mathcal{W}(p,q;Y))$. We have the following commutative diagram.
\[
\xymatrix{
  H^{s}(U_{p,q}', U_{p,q}' - \mathcal{W}(p,q;Y)) \ar[r]^-{h^{*}}
                & H^{s}(U_{p,q}, U_{p,q} - \mathcal{W}(p,q;X)) \\
  H^{s}(U_{q}', U_{q}' - \mathcal{A}(q;Y)) \ar[d] \ar[u] \ar[r]^-{h^{*}}
                & H^{s}(U_{q}, U_{q} - \mathcal{A}(q;X)) \ar[d] \ar[u]  \\
  H^{s}(U_{q}' \cap \mathcal{D}(q;Y), U_{q}' \cap \mathcal{D}(q;Y) - \{q\})
  \ar[r]^-{h^{*}}
                & H^{s}(U_{q} \cap \mathcal{D}(q;X), U_{q} \cap \mathcal{D}(q;X) - \{q\})
}
\]
All of these maps are isomorphisms. The vertical maps are induced by
inclusions. Since $h^{*} \alpha_{q}' = \alpha_{q}$, we have $h^{*}
\beta_{q}' = \beta_{q}$. Thus we get $h^{*} \beta_{p,q}' =
\beta_{p,q}$.

We also know that $h$ preserves the orientation of $U_{p,q}$. By
Lemma \ref{lemma_preserve_orientation}, the proof is completed.
\end{proof}

As in Theorem \ref{theorem_topological_equivalence}, let $h_{*}:
\overline{\mathcal{M}(p,q;X)} \rightarrow
\overline{\mathcal{M}(p,q;Y)}$, $h_{*}:
\overline{\mathcal{W}(p,q;X)} \rightarrow
\overline{\mathcal{W}(p,q;Y)}$ and $h_{*}:
\overline{\mathcal{D}(p;X)} \rightarrow \overline{\mathcal{D}(p;Y)}$
be the maps induced by $h$. Since $h$ preserves the direction of
flow, by Lemma \ref{lemma_w_orientation}, we get the following
immediately.

\begin{lemma}\label{lemma_m_orientation}
The map $h_{*}$ preserves the orientation of $\mathcal{M}(p,q;X)$.
\end{lemma}

\begin{proof}[Proof of Theorem \ref{theorem_orientation}]
Consider the map $h_{*}$ mentioned in the above. Clearly, $h_{*}$ is
identical to $h$ on $\mathcal{D}(p;X)$ and $\mathcal{W}(p,q;X)$.

By the definition of the orientation of $\mathcal{D}(p;Y)$, we know
$h_{*}$ preserves the orientation of $\mathcal{D}(p;X)$. Combining
this fact with Lemmas \ref{lemma_w_orientation} and
\ref{lemma_m_orientation}, we infer that $h_{*}$ preserves both the
boundary orientations and the product orientations. Thus $h_{*}$
preserves the orientation relations. Since these formulas are proved in the case of $Y$
in \cite[thm.\ 3.6]{qin}, we infer that the orientation
formulas are valid for $X$.
\end{proof}

\section{CW Structures}\label{section_CW}
In this section, We shall address the problem of the CW structures arising from a negative gradient-like dynamical system.

\begin{theorem}\label{theorem_cw_k(a)}
Suppose $f$ is proper and bounded below. Suppose $X$ satisfies
transversality. Suppose $a$ is a regular value of $f$. Define $K^{a}
= \bigsqcup_{f(p) \leq a} \mathcal{D}(p)$ with the topology induced
from $M$. Then $K^{a}$ is a finite CW complex with characteristic
maps $e: \overline{\mathcal{D}(p)} \rightarrow K^{a}$, and each $e$ has the explicit formula (\ref{theorem_d(p)_1}). The inclusion $K^{a}
\hookrightarrow M^{a}$ is a simple homotopy equivalence. In fact,
there is a CW decomposition of $M^{a}$ such that $K^{a}$ expands to
$M^{a}$ by elementary expansions.
\end{theorem}

\begin{theorem}\label{theorem_cw_k}
Under the assumption of Theorem \ref{theorem_cw_k(a)}, define $K =
\bigsqcup_{p \in M} \mathcal{D}(p)$. Define the topology of $K$ as
the direct limit of that of $K^{a}$ when $a$ tends to $+\infty$.
Then $K$ is a countable CW complex with characteristic maps $e:
\overline{\mathcal{D}(p)} \rightarrow K$, and each $e$ has the explicit formula (\ref{theorem_d(p)_1}). Furthermore, the inclusion $i: K
\hookrightarrow M$ is a homotopy equivalence.
\end{theorem}

As mentioned before, $\dim(\mathcal{M}(p,q)) = \mathrm{ind}(p) -
\mathrm{ind}(q) - 1$ when $p \succ q$. If $\mathrm{ind}(q) = \mathrm{ind}(p) - 1$, then
$\mathcal{M}(p,q)$ is a $0$ dimensional manifold. Actually,
$\mathcal{M}(p,q)$ consists of finitely many points because it is
compact in this case.

\begin{theorem}\label{theorem_boundary_operator}
Let $K^{a}$ (or $K$) be the CW complex in Theorem
\ref{theorem_cw_k(a)} (or \ref{theorem_cw_k}). Let $C_{*}(K^{a})$
(or $C_{*}(K)$) be the associated cellular chain complex and
$[\overline{\mathcal{D}(p)}]$ be the base element represented by the
oriented $\overline{\mathcal{D}(p)}$ in $C_{*}(K^{a})$ (or
$C_{*}(K)$). Then
\[
  \partial [\overline{\mathcal{D}(p)}] = \sum_{\mathrm{ind}(q) = \mathrm{ind}(p) -1}
  \# \mathcal{M}(p,q) [\overline{\mathcal{D}(q)}],
\]
where $\# \mathcal{M}(p,q)$ is the sum of the orientations $\pm 1$
of all points in $\mathcal{M}(p,q)$ defined in Theorem
\ref{theorem_orientation}, and $\mathrm{ind}(*)$ is the Morse index of
$*$.
\end{theorem}

\begin{remark}\label{remark_CW_compact}
Consider the special case when $M$ is compact. Theorem
\ref{theorem_cw_k(a)} shows that the compactified descending
manifolds give a CW decomposition of $M$. Before the
invention of the theory of moduli spaces, this problem was addressed
in \cite[thm.\ 1]{kalmbach2} and \cite[rem.\ 3]{laudenbach}, which
show the existence of the characteristic maps under the assumption
that the vector field is locally trivial. The paper \cite[sec.\ 4]{kalmbach1} (with a correction in \cite[sec.\ 4.5]{kalmbach2}) shows that $i^{a}: K^{a} \hookrightarrow M^{a}$ is a deformation retract. Theorem \ref{theorem_cw_k(a)} strengthens their
solutions in three ways. Firstly, the characteristic maps here $e:
\overline{\mathcal{D}(p)} \rightarrow M$ have the explicit formula
(\ref{theorem_d(p)_1}). Secondly, we drop the
assumption of the local triviality of the vector field. Thirdly, it shows that $K^{a}$ expands to $M^{a}$ by elementary expansions. In the case
when $f$ has only one critical point of index $0$, the paper
\cite[lem.\ 2.15]{barraud_cornea} also gives an answer similar to
Theorem \ref{theorem_cw_k(a)}.
\end{remark}

\begin{remark}
The above theorems show that $C_{*}(K)$ computes the homology of
$M$, and its boundary operator $\partial$ coincides with that of
Morse homology. This shows Morse homology arises from a cellular
chain complex. For Morse homology, see \cite[cor.\ 7.3]{milnor2} and
\cite{schwarz}.
\end{remark}

\begin{proof}[Proof of Theorem \ref{theorem_cw_k(a)}]
By Theorem \ref{theorem_d(p)}, $\overline{\mathcal{D}(p)}$ is a
closed disc and $e$ is continuous. Thus $K^{a}$ is a finite CW
complex with characteristic maps $e$.

We shall construct the desired CW decomposition of $M^{a}$.

Suppose $M$ is not compact. By Lemma \ref{lemma_reduction}, we can
embed $M^{a}$ into $\widetilde{M}$ and extend $f|_{M^{a}}$ to be
$\widetilde{f}$ on $\widetilde{M}$ such that
$\widetilde{f}|_{\widetilde{M} - M^{a}}
> a$. We get $\widetilde{M}^{a} = M^{a}$. As a result, we may assume
$M$ is compact.

By Theorem \ref{theorem_regular_path}, we can construct a locally
trivial field $Y$ on $M$ and a topological equivalence $h$ which
maps the orbits of $Y$ to those of $X$. Consequently,
$h(\mathcal{D}(p;Y)) = \mathcal{D}(p;X)$ and $h(K^{a}(Y)) = K^{a}$
where $K^{a}(Y) = \bigsqcup_{f(p) \leq a} \mathcal{D}(p;Y)$. By
\cite[thm. 3.8]{qin}, there exists a CW decomposition of $M^{a}$
such that $K^{a}(Y)$ expands to $M^{a}$ by elementary expansions.
Thus it suffices to prove that there exists a homeomorphism
$\widetilde{h}: M^{a} \rightarrow M^{a}$ such that $\widetilde{h}$
and $h$ coincide on $K^{a}(Y)$.

Denote by $\phi^{X}_{t}$ the flow generated by $X$ and by
$\phi^{Y}_{t}$ the flow generated by $Y$. For any $x \in f^{-1}(a)$,
we have $\phi^{Y}(- \infty, x) = r_{1}$ for some $r_{1} \in M -
M^{a}$ and $\phi^{Y}(+ \infty, x) = r_{2}$ for some $r_{2} \in
M^{a}$. Since $h$ is a topological equivalence fixing $r_{1}$ and
$r_{2}$, we have $\phi^{X}_{t}(h(x))$ is a flow line between $r_{1}$
and $r_{2}$. Thus, for any $x \in h(f^{-1}(a))$, we have $\phi^{X}(t(x), x)
\in f^{-1}(a)$ for some $t(x)$ and, by Lemma \ref{lemma_flow_intersection}, $t(x)$ is continuous on
$h(f^{-1}(a))$. Since $h(f^{-1}(a))$ is compact, there exists $T >
0$ such that $T > - t(x)$ for all $x \in h(f^{-1}(a))$. As a result,
$\phi^{X}_{T} (M^{a}) \subseteq \mathrm{Int} [h(M^{a})]$. (This is
illustrated by Figure \ref{figure_psi}, $\phi^{X}_{T} (M^{a})$ is
the shadowed part, $M^{a}$ is the part below $f^{-1}(a)$ and
$h(M^{a})$ is the part below $h(f^{-1}(a))$.) By an isotopy along
the flows generated by $X$, we can construct a homeomorphism $\psi:
h(M^{a}) \rightarrow M^{a}$ such that $\psi|_{\phi^{X}_{T} (M^{a})}
= \mathrm{Id}$ and $\psi(h(M^{a}) - \phi^{X}_{T} (M^{a})) = M^{a} -
\phi^{X}_{T} (M^{a})$. Then $\widetilde{h} = \psi \circ h$ is the
desired homeomorphism.
\end{proof}

\begin{figure}[!htbp]
\centering
\includegraphics[scale=0.4]{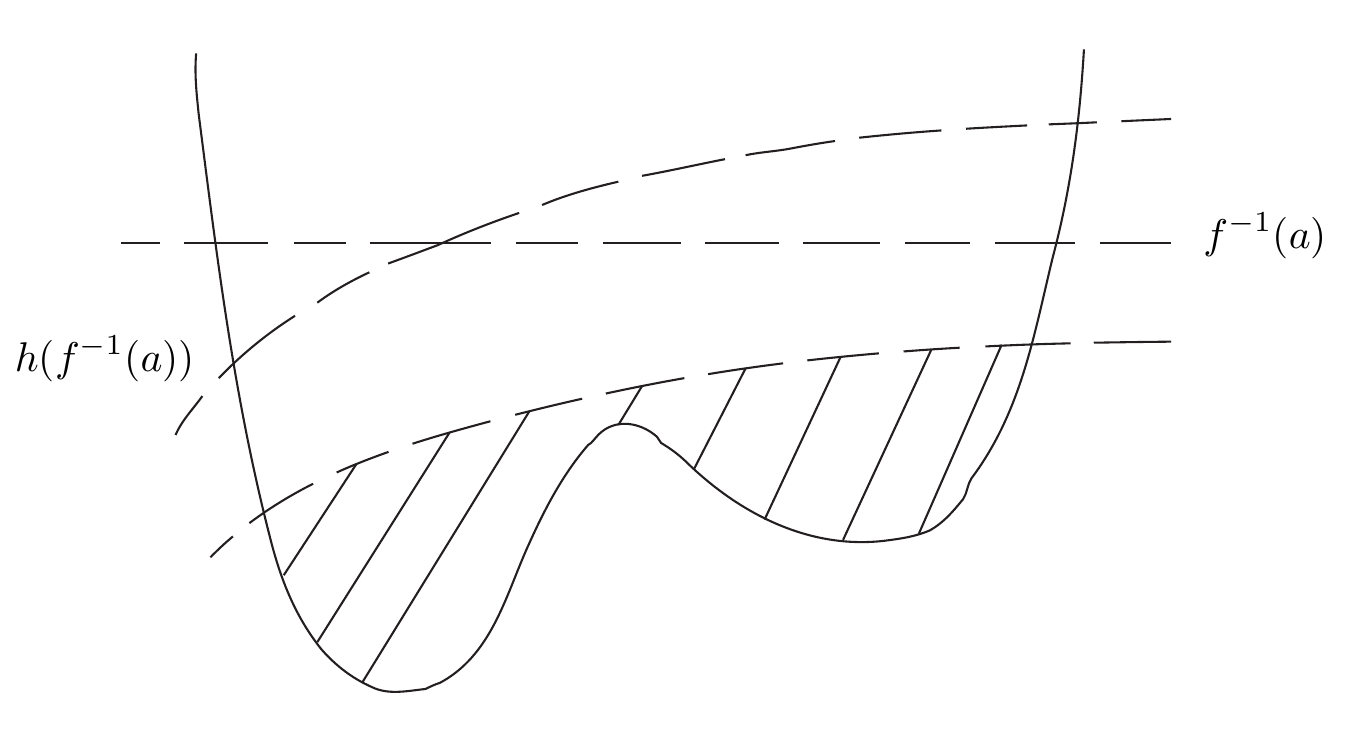} \caption{Construction of $\psi$}
\label{figure_psi}
\end{figure}

\begin{proof}[Proof of Theorem \ref{theorem_cw_k}]
The CW structure of $K$ is obvious.

By Theorem \ref{theorem_cw_k(a)}, $i: K^{a} \hookrightarrow M^{a}$
is a homotopy equivalence for any regular value $a$. Thus, it's straightforward
to check that $i: K \hookrightarrow M$ is a weak homotopy
equivalence, i.e. $i$ induces the isomorphisms between homotopy
groups. Since $M$ carries a triangulation, by Whitehead's Theorem,
$i$ is a homotopy equivalence.
\end{proof}

\begin{proof}[Proof of Theorem \ref{theorem_boundary_operator}]
There are two proofs.

First, duplicate the proof of \cite[thm.\ 3.9]{qin}. Certainly, the
local triviality of the vector field $X$ is assumed in \cite{qin}.
However, the only reason for making this assumption is that the (2)
of Theorem \ref{theorem_orientation} was proved under it in \cite{qin}. In this paper, this orientation formula is
true even if we drop this assumption. Thus, the first proof is
valid.

Second, reduce it to the case of a locally trivial vector field $Y$.
The map $h_{*}$ in Theorem \ref{theorem_topological_equivalence}
induces an isomorphism between $C_{*}(K^{a}(X))$ and
$C_{*}(K^{a}(Y))$. By Lemma \ref{lemma_m_orientation}, $h_{*}$
preserves the orientation of $\mathcal{M}(p,q;X)$. Since this
statement is true for $C_{*}(K^{a}(Y))$, the second
proof is complete.
\end{proof}

\appendix
\section{}
In this appendix, we shall prove Lemma
\ref{lemma_preserve_orientation}.

Suppose $M$ is an $n$ dimensional manifold. Suppose $L$ is a
connected and closed $k$ dimensional submanifold of $M$. Let $U$ be
a closed tubular neighborhood of $L$ such that $U$ is diffeomorphic
to a \textit{closed} disk bundle over $L$ via the exponential map.
Let $i: L \hookrightarrow U$ be the inclusion and $\pi: U
\rightarrow L$ be the smooth projection. Clearly, $i$ and $\pi$ are
proper. Thus $\pi^{*}: H^{k}_{C} (L) \rightarrow H^{k}_{C}(U)$ and
$i^{*}: H^{k}_{C} (U) \rightarrow H^{k}_{C}(L)$ are isomorphisms and
they are a pair of inverses, where $H^{*}_{C}$ is the cohomology
with compact support. Furthermore, $H^{k}_{C} (L) \cong \mathbb{Z}$,
its generator is an orientation of $L$.

Define $\displaystyle H^{n}_{C}(U, U - L) = \lim_{\overrightarrow{K
\subseteq L}} H^{n}(U, U-K)$, where $K$ is compact. We can prove the
inclusion $H^{n}(U, U-\{x\}) \rightarrow H^{n}_{C}(U, U-L)$ is an
isomorphism for any $x \in L$.

Suppose $\alpha \in H^{k}_{C}(L)$ and the Thom class $\beta \in
H^{n-k}(U, U-L)$ represent the orientation and the normal
orientation of $L$ respectively. Suppose the orientation and the
normal orientation define the orientation of $M$.

\begin{lemma}
The following cup product homomorphism is an isomorphism.
\[
\xymatrix@C=0.5cm{
  H^{k}_{C}(U) \otimes H^{n-k}(U, U-L) \ar[rr]^-{\cup}_-{\cong} && H^{n}_{C}(U, U-L). }
\]
Furthermore, for all $x \in L$, via the isomorphism $H^{n}(U, U-\{x\}) \rightarrow
H^{n}_{C}(U, U-L)$, we get $\pi^{*} \alpha \cup \beta \in
H^{n}_{C}(U, U-L)$ represents the orientation of $M$ in $H^{n}(U,
U-\{x\})$.
\end{lemma}

\begin{proof}
For any $x \in L$, we have a commutative
diagram
\[
\xymatrix{
  H^{k}(U, U-\pi^{-1}(x)) \otimes H^{n-k}(U, U-L) \ar[d]_-{\cong} \ar[r]^-{\cup}
                & H^{n}(U, U-\{x\}) \ar[d]^{\cong}  \\
  H^{k}_{C}(U)  \otimes  H^{n-k}(U, U-L) \ar[r]^-{\cup}
                & H^{n}_{C}(U, U-L)\, .             }
\]
Here the vertical maps are induced by inclusions and are isomorphisms. The horizontal
ones are given by cup product pairings. By excision and the basic
property of Thom class, we can localize the argument near $x$.
However, the disk bundle near $x$ has a product structure. Now apply
K\"{u}nneth Formula to the upper horizontal map, which completes the proof.
\end{proof}

\begin{proof}[Proof of Lemma \ref{lemma_preserve_orientation}]
It suffices to prove the special case of that $L_{i}$ is connected.

Let $U_{2}$ be a \textit{closed} tubular neighborhood of $L_{2}$
with the smooth projection $\pi_{2}: U_{2} \rightarrow L_{2}$. Let
$\alpha_{2} \in H^{k}_{C}(L_{2})$ be the orientation of $L_{2}$, by
the above lemma, we have $\pi_{2}^{*} \alpha_{2} \cup
\beta_{2}|_{U_{2}} = \gamma_{2} \in H^{n}_{C}(U_{2}, U_{2} - L_{2})$
represents the orientation of $M_{2}$ on $L_{2}$. Here
$\beta_{2}|_{U_{2}}$ is the image of $\beta_{2}$ under the inclusion
$H^{n-k}(M_{2}, M_{2} - L_{2}) \rightarrow H^{n-k}(U_{2}, U_{2} -
L_{2})$. It is the restriction of $\beta_{2}$ to $U_{2}$.

Let $U_{1}' = h^{-1}(U_{2})$. Choose a \textit{closed} tubular
neighborhood $U_{1}$ of $L_{1}$ such that $U_{1} \subseteq
\mathrm{Int} U_{1}'$ and $\pi_{1}: U_{1} \rightarrow L_{1}$ is a
smooth projection. By the above lemma again, we have the following
isomorphism
\[
\xymatrix@C=0.5cm{
  H^{k}_{C}(U_{1}) \otimes H^{n-k}(U_{1}, U_{1} - L_{1}) \ar[rr]^-{\cup}_-{\cong} && H^{n}_{C}(U_{1}, U_{1} - L_{1}),  }
\]
and
\begin{equation}\label{lemma_preserve_orientation_1}
  \pi_{1}^{*} \alpha_{1} \cup \beta_{1}|_{U_{1}} = \gamma_{1}
\end{equation}
represents the orientation of $M_{1}$ on $L_{1}$, where
$\beta_{1}|_{U_{1}}$ is the restriction of $\beta_{1}$ to $U_{1}$.

Consider the following commutative diagram:
\[
\xymatrix{
  H^{k}_{C}(U_{2}) \ar[d]_-{i_{2}^{*}} \ar[r]^-{h^{*}}
                & H^{k}_{C}(U_{1}') \ar[d]_-{j^{*}} \ar[r]^-{\iota^{*}} & H^{k}_{C}(U_{1}) \ar[dl]^-{i_{1}^{*}} \\
  H^{k}_{C}(L_{2}) \ar[r]_{h^{*}}
                & H^{k}_{C}(L_{1}) \, ,
                           }
\]
where, $i_{1}$, $i_{2}$, $j$ and $\iota$ are inclusions. Since
$h^{*} \pi_{2}^{*} \alpha_{2} \cup h^{*} \beta_{2}|_{U_{2}} = h^{*}
\gamma_{2}$, we have $\iota^{*} h^{*} \pi_{2}^{*} \alpha_{2} \cup
\iota^{*} h^{*} \beta_{2}|_{U_{2}} = \iota^{*} h^{*} \gamma_{2}$.
Since $h$ preserves the orientation of $M_{1}$ and the Thom class,
we have $\iota^{*} h^{*} \gamma_{2} = \gamma_{1}$ and $\iota^{*}
h^{*} \beta_{2}|_{U_{2}} = \beta_{1}|_{U_{1}}$. Thus
\begin{equation}\label{lemma_preserve_orientation_2}
  \iota^{*} h^{*} \pi_{2}^{*} \alpha_{2} \cup \beta_{1}|_{U_{1}} =
  \gamma_{1}.
\end{equation}
Since the cup product pairing above is an isomorphism, by
(\ref{lemma_preserve_orientation_1}) and
(\ref{lemma_preserve_orientation_2}), we infer $\iota^{*} h^{*}
\pi_{2}^{*} \alpha_{2} = \pi_{1}^{*} \alpha_{1}$. So we have
\[
   \alpha_{1}= i_{1}^{*} \pi_{1}^{*} \alpha_{1} = i_{1}^{*} \iota^{*} h^{*} \pi_{2}^{*} \alpha_{2}  = h^{*} i_{2}^{*}
  \pi_{2}^{*} \alpha_{2} = h^{*} \alpha_{2}.
\]
This completes the proof.
\end{proof}

\section*{Acknowledgements}
I wish to thank an anonymous mathematician who meticulously read through this paper and made many helpful suggestions which lead to an improved presentation of this paper. I am indebted to my PhD advisor Prof.\ John Klein for his direction, his patient educating, and his continuous encouragement. This work was partially supported by NSFC11871272.


\end{document}